\documentclass[12pt,reqno]{amsart}
\usepackage{amsthm,amsmath,amsfonts,amssymb}
\usepackage{hyperref}
\usepackage{color}
\usepackage{amsmath,amsfonts}
\usepackage{graphicx}

\newcommand{\R}{\ensuremath{\mathbb{R}}}
\newcommand{\N}{\ensuremath{\mathbb{N}}} 
\newcommand{\Z}{\ensuremath{\mathbb{Z}}} 
\newcommand{\C}{\ensuremath{\mathbb{C}}} 
\newcommand{\LL}{\ensuremath{\mathcal{L}}}
\newcommand{\delim}[3]{\left#1 #3 \right#2}  
\newcommand{\pin}[1]{\delim{\langle}{\rangle}{#1}} 
\def \ds {\displaystyle}  
\def \tn {\textnormal}  

\newcommand{\cqd}{\mbox{\framebox [1ex]{}}}
\renewcommand{\cqd}{\qed}

\newtheorem{theorem}{Theorem}[section]
\newtheorem{lemma}[theorem]{Lemma}
\newtheorem{corollary}[theorem]{Corollary}
\newtheorem{definition}[theorem]{Definition}
\newtheorem{proposition}[theorem]{Proposition}

\newtheorem{remark}[theorem]{Remark}

\title[Rate of Convergence]{Rate of Convergence of Attractors for Singularly Perturbed Semilinear Problems}

\numberwithin{equation}{section} \numberwithin{theorem}{section}

\author{} 
\begin{document}

\date{}

\maketitle

\begin{center}
{\sc A.N. Carvalho and L. Pires} 
\end{center}

%\begin{center}  
%{\fontsize{10}{10}
%\selectfont {Departamento de Matem\'atica\\
%Instituto de Ci\^encias Matem\'aticas e de Computa\c c\~ao\\
%Universidade de S\~ao Paulo \\ Caixa Postal 676, 13.565-905 S\~ao Carlos SP, Brazil
%}\\}
%\end{center}

\begin{abstract}
We exhibit a class of singularly perturbed parabolic problems which the asymptotic behavior can be described by a system of ordinary differential equation.  We estimate the convergence of attractors in the Hausdorff metric by rate of convergence of resolvent operators.  Application to spatial homogenization and large diffusion except in a neighborhood of a point will be considered.
\end{abstract}

\section{Introduction}

There are many parabolic problems whose asymptotic behavior is dictated by a system of Morse-Smale ordinary differential equations, for example, reaction diffusion equation where the diffusion coefficient become very large in all domain and reaction diffusion equation where the diffusion coefficient is very large except in a neighborhood of a finite number of points where it becomes small. These kind of problems was considered in the works \cite{Carvalho2010a,Carvalho1994,Carvalho2004,Conley1978,Fusco1987} and \cite{Hale1986}, where well-posedness, functional setting and convergence of attractors was studied. In general it is considered a family of parabolic problems depending on a positive parameter $\varepsilon$ and when $\varepsilon$ converges to zero, it is obtained a limiting ordinary differential equation that contains all dynamic of the problem. Therefore the partial differential equation that generated an abstract parabolic problem in an infinite dimensional phase space can be considered in large time behavior as an ordinary differential equation in a finite dimensional space.   

One of the most important part in these works is the study of an eigenvalue problem to determine that the first eigenvalues converge to matching eigenvalues of the limiting ODE, whereas all the others  blow up, establishing the existence of a large gap between the eigenvalues as the parameter $\varepsilon$ goes to zero. The large gap between the eigenvalues is known as gap condition. It is the main property that enable us construct an invariant manifold given as graph of certain Lipschitz map defined in a finite dimensional space that becomes flat, as the parameter varies. These manifolds are differentiable and since they are invariant we can restrict the flow to them and project the flow in a finite dimensional space. The theory of invariant manifold is well developed in many works, see \cite{Bernal1993,Carvalho2004,Henry1980} and \cite{Sell2002}.

The main property about  Morse-Smale problems is structural stability under small $C^1$ perturbation, thus, it is interesting treat the problems described above as singular perturbation of a Morse-Smale limiting ODE, therefore there is an isomorphism between the attractors in the sense that all equilibrium points and connection between them are preserved, see \cite{Bortolan}.  

We are interested in knowing how fast the dynamic above approach each other, more precisely, we want to estimate the convergence of attractors when the parameter $\varepsilon$ converges to zero. In this direction we follow the works \cite{Arrieta,Babin1992} and \cite{Santamaria2014}, where a rate of convergence for the  attractors was obtained assuming a rate of convergence for the resolvent operators. However the rates obtained in these works always show a loss with respect to rate of resolvent operators considered. In this paper we will show that for a class of parabolic problem that can be regarded as an ordinary differential equation we do not have loss in the process to pass the convergence of the resolvent operators to convergence of attractors. 

This paper is organized as follows. In the Section \ref{Functional Setting} we developed an abstract functional framework to treat the parabolic problem whose asymptotic behavior is described by a system of ordinary differential equation, we introduce the usual notation and define some notions of theory of attractor for semigroups. In the Section \ref{Compact Convergence} we introduce the important notion of compact convergence and prove the convergence of eigenvalues and eigenfunctions, we also obtain the gap condition and estimates a priori on linear semigroups that is essential to construct the invariant manifold. In the Sections \ref{Invariant Manifold}, \ref{Rate of Convergence} and \ref{Shadowing Theory and Rate of Convergence} we show that the rate of convergence of resolvent operators is the same rate of convergence of invariant manifolds and attractors, for that end, some aspects of Shadowing Theory will be presented. The Section \ref{Further Comments} is devoted to comments about more general situations that can be considered. Finally, in the Sections \ref{Applications to Spatial Homogenization} and \ref{A cell tissue reaction-diffusion problem} we consider the examples of spatial homogenization and large diffusion except in a neighborhood of a point where it becomes small.  

\section{Functional Setting}\label{Functional Setting}
In this section we introduce the general framework that we use to treat parabolic problems whose asymptotic behavior is dictated by a system of Morse-Smale ordinary differential equations. Some notion about the theory of attractors for semigroups is presented. We consider here only positive self-adjoint operator with compact resolvent.  More general situation will be addressed in the Section \ref{Further Comments}. 

Let $X_0$ be a finite dimensional Hilbert space with $\tn{dim}(X_0)=n$, for some positive integer $n$, and let $A_0:X_0\to X_0$ be an invertible  bounded linear operator whose spectrum set $\sigma(A_0)$ is given by
$$
\sigma(A_0)=\{\lambda_1^0<\lambda_2^0<\cdots<\lambda_n^0\},\quad 0<\lambda_1^0.
$$

Consider the following system of ordinary differential equation
\begin{equation}\label{unperturbed_problem}
\begin{cases}
\dot{u}^0+A_0u^0=f_0(u^0)\\
u^0(0)=u_0^0\in X_0,
\end{cases}
\end{equation}
where $f_0:X_0\to X_0$ is continuously differentiable.

Let $\{X_\varepsilon\}_{\varepsilon\in(0,\varepsilon_0]}$, $\varepsilon_0>0$, be a family of separable Hilbert spaces and let $\{A_\varepsilon\}_{\varepsilon\in(0,\varepsilon_0]}$ be a family of invertible linear operators such that, for each  $\varepsilon\in(0,\varepsilon_0]$, $A_\varepsilon:D(A_\varepsilon)\subset X_\varepsilon\to X_\varepsilon$ is self-adjoint, positive and has compact resolvent, hence its spectrum $\sigma(A_\varepsilon)$ is discrete, real and consists only eigenvalues with finite multiplicity. We denote
$$
\sigma(A_\varepsilon)=\{\lambda_1^\varepsilon<\lambda_2^\varepsilon<\cdots\},\quad 0<\lambda_1^\varepsilon,
$$
and we consider $\{\varphi_j^\varepsilon\}_{j=1}^\infty$ the associated orthonormal family of eigenfunction such that $Q_j^\varepsilon$  is the orthogonal projection onto $\tn{span}[\varphi_1^\varepsilon,\cdots,\varepsilon_j^\varepsilon]$. In this context  we can easily define the fractional powers of $A_\varepsilon$, that is, for $\alpha\in (0,1)$, $A^\alpha$ is defined by
$$
A^\alpha_\varepsilon u^\varepsilon=\sum_{j=1}^\infty (\lambda_j^\varepsilon)^\alpha Q_j^\varepsilon u^\varepsilon, \quad \forall\,u^\varepsilon\in D(A_\varepsilon^\alpha),
$$ 
where
$$
D(A_\varepsilon^\alpha)=\Big\{u\in X_\varepsilon:\sum_{j=1}^\infty (\lambda_j^\varepsilon)^{2\alpha}\| Q_j^\varepsilon u\|^2_{X_\varepsilon}<\infty \Big\}.
$$

It is well Known that $A_\varepsilon^\alpha$ is a positive self-adjoint operator and $D(A_\varepsilon^\alpha)$ becomes a Hilbert space, denoted by $X_\varepsilon^\alpha$, when equipped with the inner product $\pin{u, v}_{X_\varepsilon^\alpha}=\pin{A_\varepsilon^\alpha u,A_\varepsilon^\alpha v}_{X_\varepsilon}$, $u,v\in X_\varepsilon^\alpha$. The corresponding norm is 
$$
\|u\|_{X_\varepsilon^\alpha}=\|A_\varepsilon^\alpha u\|_{X_\varepsilon}=\Big(\sum_{j=1}^\infty(\lambda_j^\varepsilon)^{2\alpha}\| Q_j^\varepsilon u\|^2_{X_\varepsilon}  \Big)^\frac{1}{2}.
$$

Consider the following abstract parabolic problem,
\begin{equation}\label{perturbed_problem}
\begin{cases}
u^\varepsilon_t+A_\varepsilon u^\varepsilon=f_\varepsilon(u^\varepsilon)\\
u^\varepsilon(0)=u^\varepsilon_0\in X_\varepsilon^\alpha,
\end{cases}
\end{equation}
where $f_\varepsilon: X_\varepsilon^\alpha\to X_\varepsilon$ is continuously differentiable such that 
$$
\sup_{\varepsilon\in (0,\varepsilon_0]}\sup_{u\in X_\varepsilon^\alpha} \|f_\varepsilon(u)\|_{X_\varepsilon}\leq C,
$$
and 
$$
\|f_\varepsilon(u)-f_\varepsilon(v)\|_{X_\varepsilon}\leq C \|u-v\|_{X_\varepsilon^\alpha},\,\,\,\forall u,v\in X_\varepsilon^\alpha,
$$
where $C$ is a constant independent of $\varepsilon$.

Well-posedness for abstract parabolic problem type \eqref{perturbed_problem} is developed in \cite{A.N.Carvalho2010} and for ordinary differential equations type \eqref{unperturbed_problem} in \cite{Hale2009}. Basically, with some growth, signal and dissipativeness conditions, we can assume that the solution $u^0(\cdot)$ of \eqref{unperturbed_problem} and $u^\varepsilon(\cdot)$ of \eqref{perturbed_problem} are uniquely determined and defined for all $t\geq 0$.  The super index $\varepsilon$ makes reference to the elliptic operator $A_\varepsilon$ and in order to treat the problems \eqref{unperturbed_problem} and \eqref{perturbed_problem} in a coupled form we will consider $\varepsilon\in [0,\varepsilon_0]$ with the convention $X_0^\alpha=X_0$.

For each $\varepsilon\in [0,\varepsilon_0]$ we define $T_\varepsilon(t)u_0^\varepsilon=u^\varepsilon(t)$, $t> 0$ and $T_\varepsilon(0)=I_{X_\varepsilon^\alpha}$, where $I_{X_\varepsilon^\alpha}$ denotes the identity in $X_\varepsilon^\alpha$. We have $T_\varepsilon(\cdot):X_\varepsilon^\alpha\to X_\varepsilon^\alpha$ a family of nonlinear semigroups, that is, 
\begin{itemize}
\item[(i)] $T_\varepsilon(0)=I_{X_\varepsilon^\alpha}$; 
\item[(ii)] $T_\varepsilon(t+s)=T_\varepsilon(t)T_\varepsilon(s)$, $t,s\geq 0$;
\item[(iii)] $[0,\infty)\times X_\varepsilon^\alpha\ni (t,u)\to T_\varepsilon(t)u \in X_\varepsilon^\alpha$ is continuous.
\end{itemize} 
Moreover $T_\varepsilon(\cdot)$ satisfies the variation constants formula,
\begin{equation}\label{variation_constants_formula}
T_\varepsilon(t)u_0^\varepsilon=e^{-A_\varepsilon t}u_0^\varepsilon+\int_0^te^{-A_\varepsilon (t-s)}f_\varepsilon(T_\varepsilon(s)u_0^\varepsilon)\,ds,\,\,t>0,\,\,u_0^\varepsilon\in X_\varepsilon^\alpha,\,\,\, \varepsilon\in[0,\varepsilon_0],
\end{equation}
where $\{e^{-A_\varepsilon t}:t\geq 0\}$ is the linear semigroup generated by $-A_\varepsilon$. 

Since $A_0$ is a bounded operator in finite dimensional space, $e^{-A_0 t}$ is an uniformly continuous semigroup and  for $\varepsilon\in (0,\varepsilon_0]$, $e^{-A_\varepsilon t}$ is a strongly continuous semigroup given by 
$$
e^{-A_\varepsilon t}u_0^\varepsilon=\sum_{j=1}^\infty e^{-\lambda_j^\varepsilon t}Q_j^\varepsilon u_0^\varepsilon=\frac{1}{2\pi i}\int_{\Gamma} e^{\lambda t}(\lambda+A_\varepsilon)^{-1}u^\varepsilon_0\,d\lambda,\,\,t>0,
$$
where $\Gamma$ is the boundary of the set  $\{\lambda\in \C:|\tn{arg}(\lambda)|\leq \phi\}\setminus \{\lambda\in\C:|\lambda|\leq r \}$, for some $\phi \in (\pi/2,\pi)$ and $r>0$, oriented in such a way that the imaginary part is increasing.

We are interested in dynamical system that are dissipative, that is, the systems that have a compact set such that all solutions through each initial data are attracted by this set. To give meaning to the notion of attraction, we denote by $\tn{dist}_\varepsilon(A,B)$ the Hausdorff semidistance between $A,B\subset X_\varepsilon^\alpha$, $\varepsilon\in [0,\varepsilon_0]$, defined as   
$$
\tn{dist}_\varepsilon(A,B)=\sup_{a\in A}\inf_{b\in B}\|a-b\|_{X_\varepsilon^\alpha},
$$
and the symmetric Hausdorff metric
$$
\tn{d}_\varepsilon(A,B)=\max\{\tn{dist}_\varepsilon(A,B),\tn{dist}_\varepsilon(B,A)\}.
$$
\begin{definition} 
A set $\mathcal{A}_\varepsilon\subset X_\varepsilon^\alpha$ is said to be a global attractor for a semigroup $T_\varepsilon(\cdot)$ if 
\begin{itemize}
\item[(i)] $\mathcal{A}_\varepsilon$ is compact;
\item[(ii)] $\mathcal{A}_\varepsilon$ is invariant, that is, $T_\varepsilon(t)\mathcal{A}_\varepsilon=\mathcal{A}_\varepsilon$, for all $t\geq 0$;
\item[(iii)] $\mathcal{A}_\varepsilon$ attracts bounded subsets of $X_\varepsilon^\alpha$, that is, $\tn{dist}_\varepsilon(T_\varepsilon(t)B,\mathcal{A}_\varepsilon)\to 0$ as $t\to \infty$, for each $B\subset X_\varepsilon^\alpha$ bounded.
\end{itemize}
\end{definition}

We assume that the nonlinear semigroups $T_\varepsilon(\cdot)$, $\varepsilon\in [0,\varepsilon_0]$, satisfying \eqref{variation_constants_formula} have a global attractor $\mathcal{A}_\varepsilon$ uniformly bounded, that is,
$$
\sup_{u\in\mathcal{A}_\varepsilon}\|u\|_{X_\varepsilon^\alpha}\leq C,
$$ 
for some constant $C$ independent of $\varepsilon$.

A continuous function $\xi^\varepsilon(\cdot):\R\to X_\varepsilon^\alpha$, $\varepsilon\in [0,\varepsilon_0]$, is a global solution for $T_\varepsilon(\cdot)$ if it satisfies the equality $T_\varepsilon(t)\xi^\varepsilon(\tau)=\xi^\varepsilon(t+\tau)$, for all $\tau \in \R$ and $t\geq 0$. If $\xi^\varepsilon(0)=u_0^\varepsilon$ we say that $\xi^\varepsilon$ is a global solution through $u_0^\varepsilon$ and its orbit is given by $\{\xi^\varepsilon(t):t\in\R\}$.

Note that if a set $\mathcal{M}_\varepsilon\subset X_\varepsilon^\alpha$ is invariant $(T_\varepsilon(t)\mathcal{M}_\varepsilon=\mathcal{M}_\varepsilon$, $\forall\,t\geq 0)$ then any solution that starts in $\mathcal{M}_\varepsilon$ is essential for understanding the asymptotic dynamics. In fact, it is well known that any invariant set must be a union of the orbits of global solutions.

\begin{definition}
A set $\mathcal{M}_\varepsilon\subset X_\varepsilon^\alpha$ is an invariant manifold for \eqref{perturbed_problem} when for each $u_0^\varepsilon\in \mathcal{M}_\varepsilon$ there is a global solution $u^\varepsilon(\cdot)$ such that $u^\varepsilon(0)=u_0^\varepsilon$ and $u^\varepsilon(t)\in\mathcal{M}_\varepsilon$ for all $t\in\R$.
\end{definition}

Finally we also assume that the system \eqref{unperturbed_problem} generated a Morse-Smale semigroup in $X_0$. Thus, if we denote $\mathcal{E}_0$ the set of equilibrium points of $T_0(\cdot)$, then it is composed of $p$ hyperbolic points, that is,
$$
\mathcal{E}_0=\{u\in X_0:A_0 u-f_0(u)=0\}=\{u_1^{0,\ast},\dots,u_p^{0,\ast}\},
$$  
and $\sigma(A_0-f'_0(u_i^{0,\ast}))\cap\{u\in X_0:\|u\|_{X_0}=1\}=\emptyset$, for $i=1,\dots,p$. Moreover $T_0(\cdot)$ is dynamically gradient and
$$
\mathcal{A}_0=\bigcup_{i=1}^p W^u(u_i^{0,\ast}),
$$ 
where $W^u(u_i^{0,\ast})$ is the unstable manifold associated to the equilibrium point $u_i^{0,\ast}\in \mathcal{E}_0$ and for $i\neq j$ the unstable manifold $W^u(u_i^{0,\ast})$ and the local stable manifold $W_{\tn{loc}}^s(u_j^{0,\ast})$ has tranversal intersection, see \cite{Bortolan}.

\section{Compact Convergence}\label{Compact Convergence}

In this section we introduce the notion of compact convergence and we assert the necessary assumptions for estimate the convergence of attractors of \eqref{unperturbed_problem} and \eqref{perturbed_problem}. In fact, we assume that the resolvent operators and the nonlinearities have a rate of convergence and we used it to estimate the convergence of the first eigenvalues and associated eigenfunctions, consequently we also estimate the convergence of spectral projections. All  necessary conditions to construct the invariant manifold in the next section will be made here.

Assume that there are two families of bounded linear operators $E_\varepsilon:X_0\to X_\varepsilon^\alpha$ and $M_\varepsilon:X_\varepsilon\to X_0$, $\varepsilon\in (0,\varepsilon]$,  satisfying the following properties,
\begin{itemize}
\item[(i)] $M_\varepsilon\circ E_\varepsilon= I_{X_0}$;
\item[(ii)] $\|E_\varepsilon u^0\|_{X_\varepsilon}\to \|u^0\|_{X_0}$ as $\varepsilon\to 0$ ;
\item[(iii)] $\|E_\varepsilon\|_{\LL(X_0,X_\varepsilon)}$, $\|M_\varepsilon\|_{\LL(X_\varepsilon,X_0)}$,$\|E_\varepsilon\|_{\LL(X_0,X_\varepsilon^\alpha)}$, $\|M_\varepsilon\|_{\LL(X_\varepsilon^\alpha,X_0)}\leq C$, for some constant $C$ independent of $\varepsilon$.
\end{itemize}
Note that $E_\varepsilon$ is injective and $M_\varepsilon$ is surjective. Moreover the following inequalities are valid
$$
C^{-1}\|u^0\|_{X_0}\leq \|E_\varepsilon u^0\|_{X_\varepsilon^\alpha}\leq C\|u^0\|_{X_0},\quad u^0\in X_0. 
$$

\begin{definition}
We say that a family $\{u^\varepsilon\}_{\varepsilon\in (0,\varepsilon_0]}$, whith $u^\varepsilon\in X_\varepsilon^\alpha$, E-converge to $u^0\in X_0$ as $\varepsilon\to 0$ if $\|u^\varepsilon-E_\varepsilon u^0\|_{X_\varepsilon^\alpha}\to 0$ as $\varepsilon\to 0$. In that case we denote $u^\varepsilon\overset{\tn{E}}\longrightarrow u^0$.
\end{definition}

With this notion of convergence we can introduce the notion of compactness.

\begin{definition}
We say that a sequence $\{u^{\varepsilon_k}\}_{k\in\N}$, with $u^{\varepsilon_k}\in X_{\varepsilon_k}$, is relatively compact if for each subsequence $\{u^{\varepsilon_{k_l}}\}_{l\in\N}$ there is a subsequence $\{u^{\varepsilon_{k_{l_j}}}\}_{j\in\N}$ and an element $u^0\in X_0$ such that $u^{\varepsilon_{k_{l_j}}}\overset{\tn{E}}\longrightarrow u^0$. The family $\{u^\varepsilon\}_{\varepsilon\in (0,\varepsilon_0]}$, with $u^\varepsilon\in X_\varepsilon^\alpha$ is relatively compact if any subsequence $\{u^{\varepsilon_k}\}_{k\in\N}$ is relatively compact.
\end{definition}

In what follows $\{A_\varepsilon\}_{\varepsilon\in [0,\varepsilon_0]}$ is the family of operators defined in the previous section. 

\begin{definition}\label{compact_convergence}
We say that the family of bounded linear operators $\{A_\varepsilon^{-1}\}_{\varepsilon\in (0,\varepsilon_0]}$ converges compactly to $A_0^{-1}$ as $\varepsilon\to 0$, which we denote $A_\varepsilon^{-1}\overset{\tn{CC}}\longrightarrow A_0^{-1}$, if the following two conditions are satisfied:
\begin{itemize}
\item[(i)] $u^\varepsilon\overset{\tn{E}}\longrightarrow u^0\Rightarrow A_\varepsilon^{-1} u^\varepsilon\overset{\tn{E}}\longrightarrow A_0^{-1}u^0$.
\item[(ii)] each family of the form $\{A_\varepsilon^{-1}u^\varepsilon\}_{\varepsilon\in (0,\varepsilon_0]}$, with $\|u^\varepsilon\|_{X_\varepsilon^\alpha}=1$, for all $\varepsilon\in (0,\varepsilon_0]$, is relatively compact.
\end{itemize}
\end{definition}

\begin{lemma}
Assume that $A_\varepsilon^{-1}\overset{\tn{CC}}\longrightarrow A_0^{-1}$. The following assertions are true.
\begin{itemize}
\item[(i)] $\|A_\varepsilon^{-1}\|_{\LL(X_\varepsilon^\alpha)}\leq C$, for some constant $C$ independent of $\varepsilon$. 
\item[(ii)] Given any compact $K\subset \rho(-A_0)$ we have for $\varepsilon$ sufficiently small $($we still denote $\varepsilon\in (0,\varepsilon_0])$, $K\subset \rho(-A_\varepsilon)$ and 
\begin{equation}\label{unifom_bound_convergence_compact}
\sup_{\varepsilon\in(0,\varepsilon_0]}\sup_{\lambda\in K} \|(\lambda+A_\varepsilon)^{-1}\|_{\LL(X_\varepsilon,X_\varepsilon^\alpha)}\leq C,
\end{equation}  
for some constant $C$ independent of $\varepsilon$. Moreover $(\lambda+A_\varepsilon)^{-1}\overset{\tn{CC}}\longrightarrow (\lambda+A_0)^{-1}$, for all $\lambda\in K$. 
\end{itemize}
\end{lemma}
\begin{proof}
If there are sequences $\varepsilon_k\to 0$ and $u^{\varepsilon_k}\in X_{\varepsilon_k}^\alpha$ such that $\|u^{\varepsilon_k}\|_{X_{\varepsilon_k}^\alpha}=1$ and $\|A_{\varepsilon_k}^{-1} u^{\varepsilon_k}\|_{X_{\varepsilon_k}^\alpha}\to \infty$ as $\varepsilon\to 0$, then the condition (ii) in Definition \ref{compact_convergence} cannot be valid. Hence (i) is proved.

For the proof of (ii), if $K\cap \sigma(-A_\varepsilon)\neq \emptyset$, then we can take sequences $\lambda_k\in K\cap \sigma(-A_\varepsilon)$ such that $\lambda_k\to \lambda\in K$ and $u^{\varepsilon_k}\in X_{\varepsilon_k}$ such that $\|u^{\varepsilon_k}\|_{X_{\varepsilon_k}^\alpha}=1$ and $u^{\varepsilon_k}=\lambda_k A_{\varepsilon_k}^{-1}u^{\varepsilon_k}$. By condition (ii) in the Definition \ref{compact_convergence} we can assume $u^{\varepsilon_k}=\lambda_k A_{\varepsilon_k}^{-1}u^{\varepsilon_k}\overset{E}\longrightarrow \lambda u^0$, for some $u^0\in X_0$, and by condition (i) in Definition \ref{compact_convergence} we obtain $A_{\varepsilon_k}^{-1}u^{\varepsilon_k}\overset{E}\longrightarrow A_0^{-1}\lambda u^0$. It follows from uniqueness of the limit that $A_0 u^0=\lambda u^0$, that is, $\lambda\in \sigma(-A_0)\cap K$ which is an absurd. Hence $K\subset \rho(-A_\varepsilon)$.

Now, assume that \eqref{unifom_bound_convergence_compact} fails. We claim that for $\varepsilon$ sufficiently small, 
\begin{equation}\label{estimate_cc}
\sup_{\varepsilon\in(0,\varepsilon_0]}\sup_{\lambda\in K} \|(I+\lambda A_\varepsilon^{-1})^{-1}\|_{\LL(X_\varepsilon,X_\varepsilon^\alpha)}\leq C.
\end{equation}  
In fact, we have $(\lambda+A_\varepsilon)=A_\varepsilon(I+\lambda A_\varepsilon^{-1})$ which implies the existence  of $(I+\lambda A_\varepsilon^{-1})^{-1}$, and then $\tn{Kernel}(I+\lambda A_\varepsilon^{-1})=\{0\}$. Since $A_\varepsilon^{-1}$ is compact, it follows from Fredholm Alternative that  $\tn{Kernel}(I+\lambda A_\varepsilon^{-1})=\{0\}$ if only if, $\tn{R}(I+\lambda A_\varepsilon^{-1})=X_\varepsilon^{\alpha}$. Thus \eqref{estimate_cc} is equivalent to following assertion
\begin{equation}\label{estimate_cc1}
\|(I+\lambda A_\varepsilon^{-1})u^\varepsilon\|_{X_\varepsilon^\alpha}\geq \frac{1}{C},\quad u^\varepsilon\in X_\varepsilon^\alpha,\quad \|u^\varepsilon\|_{X_\varepsilon^\alpha}=1.
\end{equation}  
Assume that \eqref{estimate_cc1} is false. With the same argument above we can obtain  $u^0\in X_0$ such that $(I+\lambda A_0^{-1})u^0=0$ which is an absurd. Hence \eqref{estimate_cc} is valid and thus \eqref{unifom_bound_convergence_compact} follows noticing that
$$
\|(\lambda+A_\varepsilon)^{-1}\|_{\LL(X_\varepsilon,X_\varepsilon^\alpha)}\leq \|(I+\lambda A_\varepsilon^{-1})^{-1}\|_{\LL(X_\varepsilon,X_\varepsilon^\alpha)} \|A_\varepsilon^{-1}\|_{\LL(X_\varepsilon,X_\varepsilon^\alpha)},\quad \forall \lambda\in K,\quad \varepsilon\in (0,\varepsilon_0].
$$

Finally, we prove that $(\lambda+A_\varepsilon)^{-1}$ converges compactly to $(\lambda+A_0)^{-1}$, for all $\lambda\in K$. We denote $w^\varepsilon=(I+\lambda A_\varepsilon^{-1})^{-1}u^\varepsilon$, where $u^\varepsilon\in X_\varepsilon^\alpha$ and $\|u^\varepsilon\|_{X_\varepsilon^\alpha}=1$, thus $(\lambda+A_\varepsilon)^{-1}u^\varepsilon=A_\varepsilon^{-1}w^\varepsilon$ and $w^\varepsilon$ is uniformly bounded. Hence we can assume $(\lambda+A_\varepsilon)^{-1}u^\varepsilon\overset{E}\longrightarrow v^0\in X_0$. 
Thus, if  $u^\varepsilon\overset{E}\longrightarrow u^0$ then $A_\varepsilon^{-1}u^\varepsilon\overset{E}\longrightarrow A_0^{-1}u^0$ and 
$$
z^\varepsilon=(I+\lambda A_\varepsilon^{-1})^{-1}A_\varepsilon^{-1}u^\varepsilon=(\lambda+A_\varepsilon)^{-1}u^\varepsilon\overset{E}\longrightarrow
v^0$$
which implies
$$
A_\varepsilon^{-1}u^\varepsilon=(I+\lambda A_\varepsilon^{-1})z^\varepsilon\overset{E}\longrightarrow (I+\lambda A_0^{-1})v^0.
$$
Therefore $A_0^{-1}u^0=(I+\lambda A_0^{-1})v^0$, that is, $v^0=(\lambda+A_0)^{-1}$.
\cqd
\end{proof}

In what follows we assume $A_\varepsilon^{-1}\overset{\tn{CC}}\longrightarrow A_0^{-1}$. Moreover, we assume that there are two positive increasing function $\tau,\rho:[0,\varepsilon_0]\to [0,\infty)$ such that $\tau(0)=0=\rho(0)$,
\begin{equation}\label{rate_resolvent}
\|A_\varepsilon^{-1}-E_\varepsilon A_0^{-1}M_\varepsilon\|_{\LL(X_\varepsilon,X_\varepsilon^\alpha)}\leq \tau(\varepsilon),
\end{equation}
and
\begin{equation}\label{rate_nonlinearity}
\|f_\varepsilon(u^\varepsilon)-E_\varepsilon f_0(u^0)\|_{X_\varepsilon}\leq C \|u^\varepsilon-E_\varepsilon u^0\|_{X_\varepsilon^\alpha}+\rho(\varepsilon),\,\,\,u^\varepsilon\in X_\varepsilon^\alpha,\,\,u^0\in X_0,
\end{equation} 
where $C$ is a constant independent of $\varepsilon$.

\begin{lemma}%\label{lemma_rate_attraction} 
If $\lambda\in\rho(-A_\varepsilon)$ is such that $\lambda\notin (-\infty,-\lambda_1^0]$, then there is $\phi\in(\frac{\pi}{2},\pi)$ such that  $\lambda\in \Sigma_{-\lambda_1^0,\phi}\setminus B_r(-\lambda_1^0)=\{\mu\in \C:|\tn{arg}(\mu+\lambda_1^0)|\leq \phi\}\setminus \{\mu\in\C:|\mu+\lambda_1^0|\leq r\}$ for some small $r>0$, and it is valid the following estimate,
\begin{equation}\label{rate_resolvent_operator}
\|(\lambda+A_\varepsilon)^{-1}-E_\varepsilon(\lambda+A_0)^{-1}M_\varepsilon\|_{\LL(X_\varepsilon,X_\varepsilon^{\alpha})}\leq C\tau(\varepsilon),
\end{equation}
where $C$ is a positive constant independent of $\varepsilon$ and $\lambda$. Moreover, if a compact set $K\subset \rho(-A_\varepsilon)\cap \rho(-A_0)$ then \eqref{rate_resolvent_operator} is valid for some constant $C$ independent of $\varepsilon$ and $\lambda\in K$.
\end{lemma}
\begin{proof}
We have the following identity
$$
A_\varepsilon^{\alpha}\Big( (\lambda+A_\varepsilon)^{-1} - E_\varepsilon(\lambda+A_0)^{-1}M_\varepsilon \Big) =  A_\varepsilon(\lambda+A_\varepsilon)^{-1} A_\varepsilon^{\alpha}(A_\varepsilon^{-1}-E_\varepsilon A_0^{-1}M_\varepsilon)(\lambda+A_0)^{-1}
$$
and the result follows from \eqref{rate_resolvent} by noting that $A_\varepsilon(\lambda+A_\varepsilon)^{-1}=I-\lambda(\lambda+A_\varepsilon)^{-1}$ is uniformly bounded.
\cqd
\end{proof}

Now we will see how the convergence of the resolvent operators \eqref{rate_resolvent_operator} implies the convergence of the eigenvalues, eigenfunctions and spectral projections.
 
\begin{proposition}\label{spectral_properties} The following assertions are valid.
\begin{itemize}
\item[(i)] Given $\delta>0$ sufficiently small, for $j=1,\cdots,n$, the operators
$$
Q_\varepsilon(\lambda_j^0)=\frac{1}{2\pi i}\int_{|\xi+\lambda_j^0|=\delta} (\lambda+A_\varepsilon)^{-1}\,d\lambda,\,\,\,\varepsilon\in[0,\varepsilon_0],
$$
are compact projections on $X_\varepsilon^\alpha$ and $Q_\varepsilon(\lambda_j^0)\overset{\tn{CC}}\longrightarrow Q_0(\lambda_j^0)$. Moreover it is valid the following estimate, 
\begin{equation}\label{estimate_projection_eigenfunction}
\|Q_\varepsilon(\lambda_j^0)E_\varepsilon-E_\varepsilon Q_0(\lambda_j^0)\|_{\LL(X_0,X_\varepsilon^\alpha)}\leq C \tau(\varepsilon),
\end{equation}
for each $j=1,\cdots,n$, where $C$ is a constant independent of $\varepsilon$. 
 \item[(ii)] All projected spaces $W_\varepsilon(\lambda_j^0)=Q_\varepsilon(\lambda_j^0)X_\varepsilon^\alpha$ have the same dimension, in fact, for $j=1,\dots,n$, we have
$$
\tn{rank}(Q_\varepsilon(\lambda_j^0))=\tn{dim}(W_\varepsilon(\lambda_j^0))=\tn{dim}(W_0(\lambda_j^0))=\tn{rank}(Q_0(\lambda_j^0)),\,\,\forall\,\varepsilon\in (0,\varepsilon_0].
$$
\item[(iii)] For each $u^0\in W_0(\lambda_j^0)$ there is a sequence $\{u^{\varepsilon_k}\}$ such that $u^{\varepsilon_k}\in W_\varepsilon(\lambda_j^0)$, $\varepsilon_k\to 0$  and $u^{\varepsilon_k}\overset{\tn{E}}\longrightarrow u^0$.
\item[(iv)] Given sequences $\varepsilon_k$ and $\{u^{\varepsilon_k}\}$ such that $u^{\varepsilon_k}\in W_\varepsilon(\lambda_j^0)$ and $\|u^{\varepsilon_k}\|_{X_{\varepsilon_k}^\alpha}=1$ for all $k\in\N$, each subsequence $\{u^{\varepsilon_{k_l}}\}$ of $\{u^{\varepsilon_k}\}$ has a subsequence that converges to some $u^0\in W_0(\lambda_j^0)$.
\end{itemize}
\end{proposition}
\begin{proof}
\begin{itemize}
\item[(i)] Since $A_\varepsilon^{-1}$ is compact we have $Q_\varepsilon(\lambda_j^0)$ compact projection with finite rank (see \cite{Kato1980}). By assumption $A_\varepsilon^{-1}\overset{CC}\longrightarrow A_0^{-1}$ which implies $Q_\varepsilon(\lambda_j^0)\overset{\tn{CC}}\longrightarrow Q_0(\lambda_j^0)$. Note that for appropriated $\delta>0$ we have
$$
Q_\varepsilon(\lambda_j^0)E_\varepsilon -E_\varepsilon Q_0(\lambda_j^0)=\dfrac{1}{2\pi i } \int_{|z+\lambda_j^0|=\delta} (z+A_\varepsilon)^{-1}E_\varepsilon-E_\varepsilon (z+A_0)^{-1}\,dz,
$$
and since $M_\varepsilon E_\varepsilon=I_{X_0}$, using \eqref{rate_resolvent_operator} we obtain \eqref{estimate_projection_eigenfunction}.
\item[(ii)] If $v_1,\dots,v_l$ is a basis for $W_0(\lambda_j^0)$, then for suitably small $\varepsilon$, the set $$\{Q_\varepsilon(\lambda_j^0)E_\varepsilon v_1,\dots,Q_\varepsilon(\lambda_j^0)E_\varepsilon v_l\}$$ is linearly independent in $W_\varepsilon(\lambda_j^0)$. Thus $\tn{dim}(W_\varepsilon(\lambda_j^0))\geq \tn{dim}(W_0(\lambda_j^0))$. To prove the equality, suppose that for some sequence $\varepsilon_k\to 0$, $\tn{dim}(W_\varepsilon(\lambda_j^0))> \tn{dim}(W_0(\lambda_j^0))$. From Lemma IV.2.3 in \cite{Kato1980}, for each $k\in\N$, there is a $u^{\varepsilon_{k}}\in W_{\varepsilon_k}(\lambda_j^0)$ with $\|u^{\varepsilon_k}\|_{X_{\varepsilon_k}^\alpha}=1$ such that $\tn{dist}_\varepsilon(u^{\varepsilon_k},E_{\varepsilon_k} W_{0}(\lambda_j^0))=1$, thus by item (i) above we may assume that $u^{\varepsilon_k}=Q_{\varepsilon_k}u^{\varepsilon_k}\overset{E}\longrightarrow Q_0(\lambda_j^0)=u^0$. Thus,
$$
1\leq \|u^{\varepsilon_k}-E_\varepsilon Q_0(\lambda_j^0)u^0\|_{X_{\varepsilon_k}^\alpha}=\|Q_\varepsilon(\lambda_j^0)u^{\varepsilon_k}-E_\varepsilon Q_0(\lambda_j^0)u^0\|_{X_{\varepsilon_k}^\alpha}
\overset{\varepsilon\to 0}\longrightarrow 0,$$
which is a contradiction.
\item[(iii)] If $\varepsilon_k\to 0$ then $E_{\varepsilon_k}u^0\overset{E}\longrightarrow u^0$ and then $u^{\varepsilon_k}=Q_{\varepsilon_k}(\lambda_j^0)E_\varepsilon u^0\overset{E}\longrightarrow Q_0(\lambda_j^0)u^0$. Since $Q_0(\lambda_j^0)$ is a projection and $u^0\in W_0(\lambda_j^0)$ we have $Q_0(\lambda_j^0)u^0=u^0$, which proves the result.
\item[(iv)] Since $Q_{\varepsilon_k}(\lambda_j^0)$ is a projection we have $Q_{\varepsilon_k}u^{\varepsilon_k}=u^{\varepsilon_k}$ and follows from item (i) above that each subsequence $\{u^{\varepsilon_{k_l}}\}$ has a subsequence $\{u^{\varepsilon_{k_{l_i}}}\}$ such that $u^{\varepsilon_{k_{l_i}}}\overset{E}\longrightarrow u^0\in X_0$. Therefore $u^{\varepsilon_{k_{l_i}}}=Q_{\varepsilon_{k_{l_i}}}(\lambda_j^0)u^{\varepsilon_{k_{l_i}}}\overset{E}\longrightarrow Q_0(\lambda_j^0)u^0=u^0$, which completes the proof.  
\end{itemize}
\cqd
\end{proof}

\begin{remark}
If we choose $\delta>0$ such that $\{\lambda\in\C:|\lambda+\lambda_j^0|=\delta\}\cap \sigma(A_\varepsilon)$ is equal the eigenvalues $\lambda_1^\varepsilon,\dots,\lambda_j^\varepsilon$, then $Q_\varepsilon(\lambda_j^0)=Q_\varepsilon^j$, where $Q_\varepsilon^j$ is the orthogonal projection onto $\tn{span}[\varphi_1^\varepsilon,\dots,\varphi_j^\varepsilon]$ 
considered in the previous section. 
\end{remark}

As a consequence of Proposition \ref{spectral_properties} we have the following result. 

\begin{corollary}\label{gap_condition}
It is valid the following convergences,
\begin{itemize}
\item[(i)] $\lambda_j^\varepsilon\to \infty$ as $\varepsilon\to 0$, $j\geq n+1$.
\item[(ii)] $|\lambda_j^\varepsilon-\lambda_j^0|\leq C\tau(\varepsilon)$, $j=1,\cdots,n$,
where $C$ is a constant independent of $\varepsilon$.
\end{itemize}
\end{corollary}
\begin{proof}
We start asserting that $\lambda_j^\varepsilon\to \lambda_j^0$ as $\varepsilon\to 0$ for $j=1,\dots,n$. If this is not the case then there are $\varepsilon>0$ and sequence $\varepsilon_k\to 0$ such that 
$$
\int_{|\lambda-\lambda_j^0|=\delta}(\lambda-\lambda_j^0)^k(\lambda-A_{\varepsilon_k})^{-1}\,d\lambda=0,\quad \forall\, k\in\N.
$$
Since $A_\varepsilon^{-1}$ converges compactly to $A_0^{-1}$, we have
$$
\int_{|\lambda-\lambda_j^0|=\delta}(\lambda-\lambda_j^0)^k(\lambda-A_0)^{-1}\,d\lambda=0,\quad \forall\, k\in\N,
$$
which is an absurd because since $\lambda_j^0$ is an eigenvalue of $A_0$, $(\lambda-A_0)^{-1}$ doesn't have removable singularity at $\lambda=\lambda_j^0$.  

If (i) fails, we can take $R>0$ and sequences $\varepsilon_k\to \infty$ and $\{\lambda_j^{\varepsilon_k}\}$, $j\geq n+1$ such that $|\lambda_j^{\varepsilon_k}|\leq R$. We can assume $\lambda_j^{\varepsilon_k}\to \lambda$. Let $\varphi_j^{\varepsilon_k}$ be the corresponding eigenfunction to $\lambda_j^{\varepsilon_k}$ with $\|\varphi_j^{\varepsilon_k}\|_{X_{\varepsilon_k}^\alpha}=1$. Then $\varphi_j^{\varepsilon_k}=\lambda_j^{\varepsilon_k}A^{-1}_{\varepsilon_k}\varphi_j^{\varepsilon_k}$. Since $A_{\varepsilon_k}$ converges compactly to $A_0^{-1}$, we can assume $\varphi_j^{\varepsilon_k}\to u^0$ as $\varepsilon_k\to 0$ for some $u^0\in X_0$. Thus 
$$
\varphi_j^{\varepsilon_k}=\lambda_j^{\varepsilon_k}A^{-1}_{\varepsilon_k}\varphi_j^{\varepsilon_k}\to \lambda A^{-1}_0 u^0,
$$
as $\varepsilon_k\to 0$. Since $\varphi_j^{\varepsilon_k}\to u^0$, we get $u^0=\lambda A^{-1}_0 u^0$, which implies $\lambda \in \sigma(A_0)$, thus $\lambda=\lambda_j^0$, for some $j=1,\dots,n$, and $\lambda_j^{\varepsilon_k}\to \lambda_j^0$ as $\varepsilon_k\to 0$, $j\geq n+1$, which is an absurd.
 
Now we prove (ii). Since $\lambda_j^\varepsilon \overset{\varepsilon\to 0}\longrightarrow \lambda_j^0$, for each $j=1,\dots,n$,  then for $\varepsilon$ sufficiently small there is $u^0\in\tn{Kernel}(\lambda_j^0-A_0)$ with $\|E_\varepsilon u^0\|_{X_\varepsilon^\alpha}=1$ such that $Q_\varepsilon(\lambda_j^0)E_\varepsilon u^0$ is eigenvalue of $A_\varepsilon$ associated with $\lambda_j^\varepsilon$, thus
\begin{align*}
|\lambda_j^\varepsilon-\lambda_j^0|&\leq |\lambda_j^\varepsilon\lambda_j^0|\|E_\varepsilon A_0^{-1}Q_0(\lambda_j^0)u^0-A_\varepsilon^{-1} Q_\varepsilon(\lambda_j^0)E_\varepsilon u^0\|_{X_\varepsilon^\alpha} \\
& +|\lambda_j^0|\|(Q_\varepsilon(\lambda_j^0) E_\varepsilon-E_\varepsilon Q_0(\lambda_j^0)) u^0\|_{X_\varepsilon^\alpha},
\end{align*}
and
\begin{align*}
\|E_\varepsilon A_0^{-1}Q_0(\lambda_j^0)u^0-A_\varepsilon^{-1} Q_\varepsilon(\lambda_j^0)E_\varepsilon u^0\|_{X_\varepsilon^\alpha}&\leq \|(E_\varepsilon A_0^{-1}M_\varepsilon-A_\varepsilon^{-1})E_\varepsilon Q_0(\lambda_j^0)u^0\|_{X_\varepsilon^\alpha}
\\ 
&+\|A_\varepsilon^{-1}(E_\varepsilon Q_0(\lambda_j^0)-Q_\varepsilon(\lambda_j^0)E_\varepsilon )u^0\|_{X_\varepsilon^\alpha}.
\end{align*}
The result follows from \eqref{rate_resolvent} and \eqref{estimate_projection_eigenfunction}. 
\cqd
\end{proof}

Since we have the gap condition $\lambda_j^\varepsilon\to \infty$, as $\varepsilon\to 0$ for $j\geq n+1$ proved 
in Corollary \ref{gap_condition}, we can take $\delta>0$ and construct, for $\varepsilon$ sufficiently small (we still denote $\varepsilon\in (0,\varepsilon_0]$), the curve $\bar{\Gamma}=\Gamma_1+\Gamma_2+\Gamma_3+\Gamma_4\subset \rho(-A_\varepsilon)$, where  
$$
\Gamma_1=\{\mu\in\mathbb{C}\,;\,\tn{Re}(\mu)=-\lambda_1^0+\delta\tn{ and }|\tn{Im}(\mu)|\leq 1\},
$$
$$
\Gamma_2=\{\mu\in\mathbb{C}\,;\,-\lambda_n^0-\delta \leq\tn{Re}(\mu)\leq-\lambda_0^1+\delta\tn{ and }\tn{Im}(\mu)=1\},
$$
$$
\Gamma_3=\{\mu\in\mathbb{C}\,;\,\tn{Re}(\mu)=-\lambda_n^0-\delta\tn{ and }|\tn{Im}(\mu)|\leq 1\},
$$
$$
\Gamma_4=\{\mu\in\mathbb{C}\,;\,-\lambda_n^0-\delta \leq\tn{Re}(\mu)\leq-\lambda_1^0+\delta\tn{ and }\tn{Im}(\mu)=-1\}.
$$

For $\varepsilon\in [0,\varepsilon_0]$ we define the spectral projection
\begin{equation}\label{spectral_projection}
Q_\varepsilon=\frac{1}{2\pi i}\int_{\bar{\Gamma}} (\lambda+A_\varepsilon)^{-1}\,d\lambda.
\end{equation}
Note that $Q_0=I_{X_0}$ and $Q_\varepsilon$ coincides with the spectral projection $Q_n^\varepsilon$ previously considered and thus, the eigenspace $Q_\varepsilon X_\varepsilon^\alpha=\tn{span}[\varphi_1^\varepsilon,\cdots,\varphi_n^\varepsilon]$ is isomorphic to $\R^n$. Similarly the proof of Lemma \ref{spectral_properties} we obtain the following result.

\begin{lemma}\label{projectiom_estimate}
Let $Q_\varepsilon$ be the spectral projection defined in \eqref{spectral_projection}. There is a constant $C$ independent of $\varepsilon$ such that the following properties are valid,
\begin{itemize}
\item[(i)] $\|Q_\varepsilon\|_{\LL(X_\varepsilon,X_\varepsilon^\alpha)}\leq C;$
\item[(ii)] $\|Q_\varepsilon E_\varepsilon-E_\varepsilon\|_{\LL(X_0\,X_\varepsilon^\alpha)}\leq C \tau(\varepsilon).$
\end{itemize}
\end{lemma}

The next result play an important role in the existence of invariant manifold proved in the next section. 

\begin{lemma}\label{Linear_estimate}
Let $Q_\varepsilon$ be the spectral projection defined in \eqref{spectral_projection}. If we denote, for each $\varepsilon\in (0,\varepsilon_0],$ $Y_\varepsilon=Q_\varepsilon X_\varepsilon^\alpha$ and $Z_\varepsilon=(I-Q_\varepsilon)X_\varepsilon^\alpha$ and define the projected operators
$$
A_\varepsilon^+=A_\varepsilon|_{Y_\varepsilon}\quad\tn{and}\quad A_\varepsilon^-=A_\varepsilon|_{Z_\varepsilon},
$$
then $-A_\varepsilon^-$ generates a semigroup in $Z_\varepsilon$, $-A_\varepsilon^+$ generates a group in $Y_\varepsilon$ and the following estimates are valid,
\begin{itemize}
\item[(i)] $\|e^{-A^-_\varepsilon  t}z\|_{X_\varepsilon^\alpha}\leq Me^{-\beta t}\|z\|_{X_\varepsilon^\alpha}$, $z\in Z_\varepsilon,$ $t>0;$
\item[(ii)] $\|e^{-A^-_\varepsilon  t}z\|_{X_\varepsilon^\alpha}\leq Me^{-\beta t}t^{-\alpha}\|z\|_{X_\varepsilon}$, $z\in Z_\varepsilon,$ $t>0;$
\item[(iii)]$\|e^{-A^+_\varepsilon t}z\|_{X_\varepsilon^\alpha}\leq M e^{-\bar{\gamma} t}\|z\|_{X_\varepsilon^\alpha}$, $z\in Y_\varepsilon$ , $t>0$;
\item[(iv)]$\|e^{-A_\varepsilon^+ t}z\|_{X_\varepsilon^\alpha}\leq M e^{-\gamma t}\|z\|_{X_\varepsilon^\alpha},$ $z\in Y_\varepsilon,$ $t< 0;$
\item[(v)]$\|E_\varepsilon e^{-A_0 t}M_\varepsilon z\|_{X_\varepsilon^\alpha}\leq M e^{-\gamma t}\|z\|_{X_\varepsilon^\alpha},$ $z\in X_\varepsilon^\alpha,$ $t< 0;$
\item[(vi)]$\|e^{-A_\varepsilon^+ t}z-E_\varepsilon e^{-A_0 t}M_\varepsilon z\|_{X^\alpha_\varepsilon}\leq Me^{-\gamma t}\tau(\varepsilon) \|z\|_{X_\varepsilon^\alpha},$ $z\in Y_\varepsilon$, $t< 0;$
\item[(vii)] $\|e^{-A_\varepsilon^+ t}z-E_\varepsilon e^{-A_0 t}M_\varepsilon z\|_{X^\alpha_\varepsilon}\leq M\tau(\varepsilon)\|z\|_{X_\varepsilon^\alpha}$, $z\in Y_\varepsilon$, $t>0$,
\end{itemize}
where $\beta=\lambda_{n+1}^\varepsilon\to \infty$ as $\varepsilon\to 0$, $\bar{\gamma}=\lambda_1^\varepsilon$, $\gamma=\lambda^0_n+\delta$ and $M$ is a constant independent of $\varepsilon$. 
\end{lemma}
\begin{proof}
\begin{itemize}
\item[(i)] We have for $z\in Z_\varepsilon$,
$$
e^{-A_\varepsilon^- t}z=\sum_{i=n+1}^\infty e^{-\lambda_i^\varepsilon t}\pin{z,\varphi_i^\varepsilon}_{X_\varepsilon}\varphi_i^\varepsilon,\quad t>0,
$$
but if $i\geq n+1$, then  $\lambda_{n+1}^\varepsilon\leq\lambda_i^\varepsilon$ which implies $e^{-\lambda_i^\varepsilon t}\leq e^{\lambda_{n+1}^\varepsilon t}$ for $t>0$. Thus
$$
\|e^{-A^-_\varepsilon  t}z\|_{X_\varepsilon^\alpha}\leq \Big(e^{-2\lambda_{n+1}^\varepsilon t}\sum_{i=n+1}^\infty \pin{z,\varphi_i^\varepsilon}^2_{X_\varepsilon}(\lambda_i^\varepsilon)^{2\alpha}\Big)^\frac{1}{2}\leq M e^{-\lambda_{n+1}^\varepsilon t}\|z\|_{X_\varepsilon^\alpha},\quad t>0.
$$
\item[(ii)]The function $f(\mu)=e^{-\mu t}\mu^\alpha$ attains its maximum at $\mu=\alpha/t$, $t>0$. Then,
$$
\|e^{-A^-_\varepsilon  t}z\|_{X_\varepsilon^\alpha}\leq \begin{cases} e^{-\lambda_{n+1}^\varepsilon t}(\lambda_{n+1}^\varepsilon)^\alpha\|z\|_{X_\varepsilon}, \,\,\,\alpha/t<\lambda_{n+1}^\varepsilon, \\ e^{-\lambda_{n+1}^\varepsilon t}(\alpha/t)^\alpha \|z\|_{X_\varepsilon},\,\,\,\alpha/t>\lambda_{n+1}^\varepsilon.\end{cases}
$$
\item[(iii)] We have for $z\in Y_\varepsilon$,
$$
e^{-A_\varepsilon^+ t}z=\sum_{i=1}^n e^{-\lambda_i^\varepsilon t}\pin{z,\varphi_i^\varepsilon}_{X_\varepsilon}\varphi_i^\varepsilon,
$$
but if $i\leq n+1$, then $\lambda_{1}^\varepsilon\leq\lambda_i^\varepsilon$ which implies $e^{-\lambda_i^\varepsilon t}\leq e^{-\lambda_{1}^\varepsilon t}$ for $t>0$. Thus
$\|e^{-A^+_\varepsilon  t}z\|_{X_\varepsilon^\alpha}\leq Me^{-\lambda_{1}^\varepsilon t}\|z\|_{X_\varepsilon^\alpha}$, $t>0.$
\item[(iv)] Similarly to item (iii) we obtain $\|e^{-A^+_\varepsilon  t}z\|_{X_\varepsilon^\alpha}\leq Me^{-\lambda_{n}^\varepsilon t}\|z\|_{X_\varepsilon^\alpha}$, $t<0.$ But $\lambda_n^\varepsilon\to \lambda_n^0<\lambda_n^0+\delta$, thus we can take $\varepsilon$ sufficiently small such that $e^{-\lambda_n^\varepsilon}\leq e^{-(\lambda_n^0+\delta)t},$ $t<0$.
\item[(v)] Similarly to item (iii) we also obtain $\|E_\varepsilon e^{-A_0 t}M_\varepsilon z\|_{X_\varepsilon^\alpha}\leq Me^{-(\lambda_{n}^0+\delta)t}\|z\|_{X_\varepsilon^\alpha}$, $t<0,$ where we have used that $E_\varepsilon$  and $M_\varepsilon$ are uniformly bounded.
\item[(vi)] Since $A_\varepsilon^{+}$ is a bounded operator in the finite dimensional space, for $z\in Y_\varepsilon$ and $t<0$, we  have
\begin{align*}
\|e^{-A_\varepsilon^+ t}z-E_\varepsilon e^{-A_0 t}M_\varepsilon z\|_{X_\varepsilon^\alpha} & \leq\frac{1}{2\pi}\int_{\bar{\Gamma}} \|e^{\mu t}((\mu+A_\varepsilon)^{-1}z-E_\varepsilon(\mu+A_0)^{-1}M_\varepsilon z)\|_{X_\varepsilon^\alpha}\,d\mu\\
& \leq C\tau(\varepsilon) \int_{\bar{\Gamma}} |e^{\mu t}|\,|d\mu| \|z\|_{X_\varepsilon^\alpha}\leq C\tau(\varepsilon) \sup_{\mu\in\bar{\Gamma}} e^{\tn{Re}(\mu) t} \|z\|_{X_\varepsilon^\alpha}\\
& \leq M e^{-(\lambda_n^0+\delta) t}\tau(\varepsilon) \|z\|_{X_\varepsilon^\alpha},
\end{align*}
where we have used $X_\varepsilon^\alpha\subset X_\varepsilon$ uniformly.
\item[(vii)] Analogous to item (vi) using that $\bar{\Gamma}$ is compact and $e^{\mu t}$ is continuous.
\end{itemize}
\cqd
\end{proof}

\section{Invariant Manifold}\label{Invariant Manifold}

In this section we use the gap condition in the Corollary \ref{gap_condition} and the estimates in the Lemma \ref{Linear_estimate} to construct an invariant manifolds for the problems \eqref{perturbed_problem}. We  prove that these manifolds converge to $X_0$ as $\varepsilon\to 0$ and this convergence can be estimate by $\tau(\varepsilon)+\rho(\varepsilon)$. 

\begin{theorem}\label{estimate_invariant_manifold}
For $\varepsilon$ sufficiently small there is an invariant manifold $\mathcal{M}_\varepsilon$ for \eqref{perturbed_problem}, which is given by graph of a certain Lipschitz continuous map $s_\ast^\varepsilon:Y_\varepsilon\to Z_\varepsilon$ as 
$$
\mathcal{M}_\varepsilon=\{u^\varepsilon\in X_\varepsilon^\alpha\,;\, u^\varepsilon = Q_\varepsilon u^\varepsilon+s_{*}^\varepsilon(Q_\varepsilon u^\varepsilon)\}
.$$ 
The map $s_\ast^\varepsilon:Y_\varepsilon\to Z_\varepsilon$ satisfies the condition
\begin{equation}\label{estimate_invariant_manifold_reaction_diffusion}
|\!|\!|s_\ast^\varepsilon |\!|\!|=\sup_{v^\varepsilon\in Y_\varepsilon}\|s_\ast^\varepsilon(v^\varepsilon)\|_{X_\varepsilon^\alpha}\leq C (\tau(\varepsilon)+\rho(\varepsilon)),
\end{equation}
for some positive constant $C$ independent of $\varepsilon$.
The invariant manifold $\mathcal{M}_\varepsilon$ is exponentially attracting and the global attractor $\mathcal{A}_\varepsilon$ of the problem \eqref{perturbed_problem} lying in $\mathcal{M}_\varepsilon$ and the flow on $\mathcal{A}_\varepsilon$ is given by
$$
u^\varepsilon(t)=v^\varepsilon(t)+s_\ast^\varepsilon(v^\varepsilon(t)), \quad t\in\R,
$$ 
where $v^\varepsilon(t)$ satisfy
$$
\dot{v^\varepsilon}+A_\varepsilon^+v^\varepsilon=Q_\varepsilon f_\varepsilon(v^\varepsilon+s_\ast^\varepsilon(v^\varepsilon(t))).
$$
\end{theorem}
\begin{proof}
Given $L,\Delta>0$ we consider the set  
$$
\Sigma_\varepsilon = \Big\{s^\varepsilon: Y_\varepsilon\to Z_\varepsilon\,;\, |\!|\!| s^\varepsilon|\!|\!|\leq D  \tn{ and }\|s^\varepsilon(v)-s^\varepsilon(\tilde{v})\|_{X_\varepsilon^{\alpha}}\leq \Delta \|v-\tilde{v}\|_{Y_\varepsilon}\Big\}.
$$
Thus $(\Sigma_\varepsilon, |\!|\!| \cdot |\!|\!|)$ is a complete metric space. We write the solution $u^\varepsilon$ of \eqref{perturbed_problem} as $u^\varepsilon=v^\varepsilon+z^\varepsilon$, with $v^\varepsilon\in Y_\varepsilon$ and $z^\varepsilon\in Z_\varepsilon$ and since $Q_\varepsilon$ and $I-Q_\varepsilon$ commute with $A_\varepsilon$,  we can write \eqref{perturbed_problem} in the coupled form 
\begin{equation}\label{couple_system}
\begin{cases}
v_t^\varepsilon+A_\varepsilon^+ v^\varepsilon = Q_\varepsilon f_\varepsilon(v^\varepsilon+z^\varepsilon):=H_\varepsilon(v^\varepsilon,z^\varepsilon)\\
z_t^\varepsilon+A_\varepsilon^- z^\varepsilon = (I-Q_\varepsilon)f_\varepsilon(v^\varepsilon+z^\varepsilon):=G_\varepsilon(v^\varepsilon,z^\varepsilon).
\end{cases}
\end{equation}
Since $f_\varepsilon$ is continuously differentiable we can choose  $\rho>0$ such that for all $ v^\varepsilon,\tilde{v}^\varepsilon\in Y_\varepsilon$ and $z^\varepsilon,\tilde{z}^\varepsilon\in Z_\varepsilon$, we have
\begin{align}\label{est_man_1}
&\|H_\varepsilon( v^\varepsilon,z^\varepsilon)\|_{X^\alpha_\varepsilon}\leq \rho,\quad\|G_\varepsilon( v^\varepsilon,z^\varepsilon)\|_{X_\varepsilon^\alpha}\leq \rho,\nonumber \\
&\|H_\varepsilon( v^\varepsilon,z^\varepsilon)-H_\varepsilon( \tilde{v}^\varepsilon,\tilde{z}^\varepsilon)\|_{X^\alpha_\varepsilon}\leq \rho (\|v^\varepsilon-\tilde{v}_\varepsilon\|_{X^\alpha_\varepsilon}+\|z^\varepsilon-\tilde{z}_\varepsilon \|_{X_\varepsilon^\alpha}),\\
&\|G_\varepsilon( v^\varepsilon,z^\varepsilon)-G_\varepsilon( \tilde{v}^\varepsilon,\tilde{z}^\varepsilon)\|_{X_\varepsilon^\alpha}\leq \rho (\|v^\varepsilon-\tilde{v}_\varepsilon\|_{X_\varepsilon^\alpha} +\|z^\varepsilon-\tilde{z}_\varepsilon \|_{X_\varepsilon^\alpha}),\nonumber
\end{align}
and for $\varepsilon$ sufficiently small, we can take
\begin{align}\label{est_man_2}
& \ds\rho M\beta^{-1}\leq D,\quad 0 \leq \beta-\gamma-\rho M(1+\Delta),\quad 0 \leq \beta-\gamma-\rho M,\nonumber\\
& \ds\frac{\rho M^2(1+\Delta)}{\beta-\gamma-\rho M(1+\Delta)}\leq \Delta,\quad \ds\rho M \beta^{-1}+\frac{\rho^2M^2(1+\Delta)\gamma^{-1}}{\beta-\gamma-\rho M(1+\Delta)} \leq \frac{1}{2},\\
& \ds\rho M \beta^{-1}+\frac{\rho^2M^2\gamma^{-1}}{\beta-\gamma-\rho M} <1,\quad \ds L=\Big[\rho M+\frac{\rho^2 M^2(1+\Delta)(1+M)}{\beta-\gamma-\rho M(1+\Delta)} \Big],\,\, \beta-L> 0\nonumber,
\end{align}
where $\beta$, $\gamma$ and $M$ are given in the Lemma \ref{Linear_estimate}. Note that, since we have the gap condition $\beta=\beta(\varepsilon)=\lambda^\varepsilon_{n+1}\to \infty$ as $n\to\infty$, the estimates \eqref{est_man_2} are satisfied.   

Let $s^\varepsilon\in \Sigma_\varepsilon$ and $v^\varepsilon(t)=v^\varepsilon(t,\tau,\eta,s^\varepsilon)$ be the solution of 
$$
\begin{cases}
v_t^\varepsilon+A_\varepsilon^+ v^\varepsilon=H_\varepsilon(v^\varepsilon,s^\varepsilon(v^\varepsilon)),\quad  t<\tau \\
v^\varepsilon(\tau)=\eta\in Y_\varepsilon.
 \end{cases}
$$
We define $\Phi_\varepsilon: \Sigma_\varepsilon\to\Sigma_\varepsilon$ given by
$$
\Phi_\varepsilon(s^\varepsilon)(\eta)=\int_{-\infty}^{\tau} e^{-A_\varepsilon^- (\tau - r)}G_\varepsilon(v^\varepsilon(r),s^\varepsilon(v^\varepsilon(r)) )\,dr.
$$
Note that $G_\varepsilon$ and $H_\varepsilon$ are maps acting in $X_\varepsilon^\alpha$, then by \eqref{est_man_1}, \eqref{est_man_2} and Lemma \ref{Linear_estimate}, we have 
$$
\|\Phi_\varepsilon(s^\varepsilon)(\eta)\|_{X_\varepsilon^{\alpha}}\leq \rho M \int_{-\infty}^\tau e^{-\beta (\tau-r)}\,dr=\rho M\beta^{-1} \leq D.
$$ 
Also for $s^\varepsilon, \tilde{s^\varepsilon}\in \Sigma_\varepsilon$, $\eta,\tilde{\eta}\in Y_\varepsilon$,  $v^\varepsilon(t)=v^\varepsilon(t,\tau,\eta,s^\varepsilon)$ and $\tilde{v}^\varepsilon(t)=\tilde{v}^\varepsilon(t,\tau,\tilde{\eta},\tilde{s}^\varepsilon)$ we have
\begin{multline*}
v^\varepsilon(t)-\tilde{v}^\varepsilon(t) = e^{-A_\varepsilon^+ (t-\tau)}(\eta-\tilde{\eta}) \\ +\int_{\tau}^t   e^{-A_\varepsilon^+ (t - r)}[H_\varepsilon(v^\varepsilon(r),s^\varepsilon(v^\varepsilon(r)) )-H_\varepsilon(\tilde{v}^\varepsilon(r),\tilde{s}^\varepsilon(\tilde{v}^\varepsilon(r)) )] \,dr,
\end{multline*}
and we can prove using Gronwall's inequality that
$$
\|v^\varepsilon(t)-\tilde{v}^\varepsilon(t)\|_{X_\varepsilon^\alpha}\leq \Big[M\|\eta-\tilde{\eta}\|_{X_\varepsilon^\alpha}+\rho M \gamma^{-1}|\!|\!| s^\varepsilon-\tilde{s}^\varepsilon |\!|\!|  \Big]  e^{[\rho M(1+\Delta)+\gamma](\tau-t)}.
$$
From this we obtain 
\begin{multline*}
\|\Phi_\varepsilon(s^\varepsilon)(\eta)-\Phi_\varepsilon(\tilde{s}^\varepsilon)(\tilde{\eta})\|_{X_\varepsilon^{\alpha}}\\ \leq \Big[\frac{\rho M^2(1+\Delta)}{\beta-\gamma-\rho M(1+\Delta)}\Big] \|\eta-\tilde{\eta}\|_{X_\varepsilon^\alpha}+\Big[\rho M\beta^{-1}+\frac{\rho^2M^2(1+\Delta)\gamma^{-1}}{\beta-\gamma-\rho M (1+\Delta)}\Big]|\!|\!|s^\varepsilon-\tilde{s}^\varepsilon |\!|\!|.
\end{multline*}
Therefore $\Phi_\varepsilon$ is a contraction on $\Sigma_\varepsilon$ and there is a unique  $s_\ast^\varepsilon \in \Sigma_\varepsilon$. 

Let $(\bar{v}^\varepsilon,\bar{z}^\varepsilon)\in \mathcal{M}_\varepsilon$, $\bar{z}^\varepsilon=s_\ast^\varepsilon(\bar{v}^\varepsilon)$ and let $v_{s_\ast}^\varepsilon(t)$ be the solution of 
$$
\begin{cases}
v_t^\varepsilon+A_\varepsilon^+ v^\varepsilon=H_\varepsilon(v^\varepsilon,s_*^\varepsilon(v^\varepsilon)),\quad  t<\tau \\
v^\varepsilon(0)=\bar{v}^\varepsilon.
 \end{cases}
$$ 
Thus, $\{(v_{s_*}^\varepsilon(t), s_*^\varepsilon(v_{s_*}^\varepsilon(t))\}_{t\in\R}$ defines a curve on $\mathcal{M}_\varepsilon$. But the only solution of equation 
$$
z_t^\varepsilon+A_\varepsilon^- z^\varepsilon=G_\varepsilon(v_{s_*}^\varepsilon(t),s_*^\varepsilon(v_{s_*}^\varepsilon(t)))
$$ 
which stay bounded when $t\to-\infty$ is given by
$$
z_{s_*}^\varepsilon=\int_{-\infty}^t e^{-A_\varepsilon^- (t-r)}G_\varepsilon(v_{s_*}^\varepsilon(t),s_*^\varepsilon(v_{s_*}^\varepsilon(t)))\,dr = s_*^\varepsilon(v_{s_*}^\varepsilon(t)).
$$
Therefore $(v_{s_*}^\varepsilon(t), s_*^\varepsilon(v_{s_*}^\varepsilon(t))$ is a solution of \eqref{perturbed_problem} through $(\bar{v}^\varepsilon,\bar{z}^\varepsilon)$ and thus $\mathcal{M}_\varepsilon$ is a invariant manifold.

Now we will prove the estimate \eqref{estimate_invariant_manifold_reaction_diffusion}. We have $Q_0=I_{X_0}$ which implies $I-Q_0=0$, therefore we can assume $\mathcal{M}_0=X_0$, $s_\ast^0=0$ and $G_0(v^0,s_*^0(v^0))=0$, where $v^0$ is solution of \eqref{unperturbed_problem} with $v^0(0)=M_\varepsilon\eta$ and $\|\eta\|_{X_\varepsilon^\alpha}\leq C$, for some constant $C$ independent of $\varepsilon$. Thus 
\begin{align*}
\|s^\varepsilon_*(\eta)\|_{X_\varepsilon^\alpha}& \leq \int_{-\infty}^\tau \|e^{-A_\varepsilon^-(\tau-r)}G_\varepsilon(v^\varepsilon,s_*^\varepsilon(v^\varepsilon))\|_{X_\varepsilon^\alpha}\,dr\\
& \leq \int_{-\infty}^\tau \|e^{-A_\varepsilon^-(\tau-r)}G_\varepsilon(v^\varepsilon,s_*^\varepsilon(v^\varepsilon))-e^{-A_\varepsilon^-(\tau-r)}G_\varepsilon(E_\varepsilon v^0,0)\|_{X_\varepsilon^\alpha}\,dr\\
& + \int_{-\infty}^\tau \|e^{-A_\varepsilon^-(\tau-r)}G_\varepsilon(E_\varepsilon v^0,0)\|_{X_\varepsilon^\alpha}\,dr.
\end{align*}
If we denote the last two integrals for $I_1$ and $I_2$ respectively, with the aid of \eqref{est_man_1}, \eqref{est_man_2} and Lemma \ref{Linear_estimate}, we get
\begin{align*}
I_1&\leq \int_{-\infty}^\tau M e^{-\beta (\tau - r)}\rho [\|v^\varepsilon-E_\varepsilon v^0\|_{X_\varepsilon^\alpha}+|\!|\!|s_*^\varepsilon|\!|\!|]\,dr\\
&\leq \rho M\int_{-\infty}^\tau e^{-\beta (\tau - r)}\|v^\varepsilon-E_\varepsilon v^0\|_{X_\varepsilon^\alpha}\,dr \\
& +\rho M |\!|\!|s_*^\varepsilon|\!|\!|\int_{-\infty}^\tau  e^{-\beta (\tau - r)}\,dr\\
&= \rho M\int_{-\infty}^\tau e^{-\beta (\tau-r)}\|v^\varepsilon-E_\varepsilon v^0\|_{X_\varepsilon^\alpha}\,dr +\rho M\beta^{-1} |\!|\!|s_*^\varepsilon|\!|\!|.
\end{align*}
For $I_2$, observe that
$$
G_\varepsilon(E_\varepsilon v^0,0)=(I-Q_\varepsilon)f_\varepsilon(E_\varepsilon v^0)=(I-Q_\varepsilon)[f_\varepsilon(E_\varepsilon v^0)-E_\varepsilon f_0(v^0)] + (I-Q_\varepsilon)E_\varepsilon f_0(v^0),
$$
thus by \eqref{rate_nonlinearity} and Lemma \eqref{projectiom_estimate}, we have
$
I_2\leq C (\tau(\varepsilon)+\rho(\varepsilon)),
$
for some constant $C$ independent of $\varepsilon$. Therefore
\begin{align*}
\|s^\varepsilon_*(\eta)\|_{X_\varepsilon^{\alpha}} & \leq C(\tau(\varepsilon)+\rho(\varepsilon))+\rho M\beta^{-1} |\!|\!|s_*^\varepsilon|\!|\!| + \rho M\int_{-\infty}^\tau e^{-\beta(\tau-r)}\|v^\varepsilon-E_\varepsilon v^0\|_{X_\varepsilon^\alpha}\,dr.
\end{align*}
But, for $t<\tau$, we have
$$
v^\varepsilon(t)=e^{-A_\varepsilon^+ (t-\tau)}\eta+\int_{t}^\tau e^{-A_\varepsilon^+ (t-r)}H_\varepsilon(v^\varepsilon(r),s_\ast^\varepsilon(v^\varepsilon(r)))\,dr,
$$ 
$$
v^0(t)=e^{-A_0(t-\tau)}M_\varepsilon \eta+\int_{t}^\tau e^{-A_0(t-r)}f_0(v^0(r))\,dr,
$$ 
then
\begin{multline*}
\|v^\varepsilon(t)-E_\varepsilon v^0(t)\|_{X_\varepsilon^\alpha} \leq \|(e^{-A^+_\varepsilon (t-\tau)}-E_\varepsilon e^{-A_0 (t-\tau)}M_\varepsilon)\eta\|_{X_\varepsilon^\alpha} \\ +\int_t^{\tau} \|e^{-A_\varepsilon^+(t-r)}H_\varepsilon(v^\varepsilon(r),s_\ast^\varepsilon(v^\varepsilon(r)))-E_\varepsilon e^{-A_0(t-r)}f_0(v^0(r))\|_{X_\varepsilon^\alpha}\,dr.
\end{multline*}
We can write
\begin{align*}
\int_t^{\tau} \|& e^{-A_\varepsilon^+(t-r)} H_\varepsilon(v^\varepsilon(r),s_\ast^\varepsilon(v^\varepsilon(r)))-E_\varepsilon e^{-A_0(t-r)}f_0(v^0(r))\|_{X_\varepsilon^\alpha}\,dr \\
&\leq \int_t^{\tau} \|e^{-A_\varepsilon^+(t-r)}H_\varepsilon(v^\varepsilon(r),s_\ast^\varepsilon(v^\varepsilon(r)))-E_\varepsilon e^{-A_0(t-r)}M_\varepsilon H_\varepsilon(v^\varepsilon(r),s_\ast^\varepsilon(v^\varepsilon(r)))\|_{X_\varepsilon^\alpha}\,dr \\
& + \int_t^{\tau} \|E_\varepsilon e^{-A_0(t-r)}M_\varepsilon [H_\varepsilon(v^\varepsilon(r),s_\ast^\varepsilon(v^\varepsilon(r)))-H_\varepsilon(E_\varepsilon v^0(r),0)]\|_{X_\varepsilon^\alpha}\,dr\\ 
& + \int_t^{\tau} \|E_\varepsilon e^{-A_0(t-r)}M_\varepsilon [H_\varepsilon(E_\varepsilon v^0(r),0)-E_\varepsilon  f_0(v^0(r))]\|_{X_\varepsilon^\alpha}\,dr.
\end{align*}
If we denote $\varphi(t)=\|v^\varepsilon(t)-E_\varepsilon v^0(t)\|_{X_\varepsilon^\alpha}e^{\gamma (t-\tau)}$, by \eqref{rate_nonlinearity}, \eqref{est_man_1}, \eqref{est_man_2} and the Lemmas \ref{projectiom_estimate} and \ref{Linear_estimate}, we obtain 
$$
\varphi(t)\leq C(\tau(\varepsilon)+\rho(\varepsilon))+\rho M \gamma^{-1}|\!|\!|s_*^\varepsilon|\!|\!|+\rho M\int_t^\tau \varphi(r)\,dr,
$$
which implies by Gronwall's inequality
$$
\|v^\varepsilon(t)-E_\varepsilon v^0(t)\|_{X_\varepsilon^\alpha}\leq [C(\tau(\varepsilon)+\rho(\varepsilon))+\rho M\gamma^{-1}|\!|\!|s_*^\varepsilon|\!|\!|]e^{(\rho M+\gamma)(\tau-t)},
$$
thus
\begin{align*}
&\|s_*^\varepsilon(\eta)\|_{X_\varepsilon^\alpha} \leq C(\tau(\varepsilon)+\rho(\varepsilon))+[\rho M\beta^{-1}+\frac{\rho^2 M^2 \gamma^{-1}}{\beta-\gamma-\rho M}]|\!|\!|s_*^\varepsilon|\!|\!|.
\end{align*}
It follows from \eqref{est_man_2} that  $|\!|\!|s_*^\varepsilon|\!|\!|\leq C(\tau(\varepsilon)+\rho(\varepsilon)).$

It remains that $\mathcal{M}_\varepsilon$ is exponentially attracting and $\mathcal{A}_\varepsilon\subset\mathcal{M}_\varepsilon$. Let $(v^\varepsilon,z^\varepsilon)\in Y_\varepsilon\oplus Z_\varepsilon$ be the solution of \eqref{couple_system} and define $\xi^\varepsilon(t)=z^\varepsilon-s_*^\varepsilon(v^\varepsilon(t))$ and consider $y^\varepsilon(r,t), r\leq t$, $t\geq 0$, the solution of 
$$
\begin{cases}
y_t^\varepsilon+A_\varepsilon^+ y^\varepsilon=H_\varepsilon(y^\varepsilon,s_*^\varepsilon(y^\varepsilon)),\quad  r\leq t \\
y^\varepsilon(t,t)=v^\varepsilon(t).
\end{cases}
$$
Thus,
\begin{align*}
\|y^\varepsilon(r,t)-&v^\varepsilon(r)\|_{X_\varepsilon^\alpha} \\
& =\Big\|\int_t^r e^{-A_\varepsilon^+(r-\theta)}[H_\varepsilon(y^\varepsilon(\theta,t),s_*^\varepsilon(y^\varepsilon(\theta,t)))-H_\varepsilon(v^\varepsilon(\theta),z^\varepsilon
(\theta))]\,d\theta\Big\|_{X_\varepsilon^\alpha} \\
&\leq \rho M \int_r^t e^{-\gamma(r-\theta)}[(1+\Delta)\|y^\varepsilon(\theta,t)-v^\varepsilon(\theta)\|_{X_\varepsilon^\alpha}+\|\xi^\varepsilon(\theta)\|_{X_\varepsilon^{\alpha}}]\,d\theta.
\end{align*}
By Gronwall's inequality
$$
\|y^\varepsilon(r,t)-v^\varepsilon(r)\|_{X_\varepsilon^\alpha}\leq\rho M \int_r^{t} e^{-(-\gamma-\rho M(1+\Delta))(\theta-r)} \|\xi^\varepsilon(\theta)\|_{X_\varepsilon^{\alpha}}\,d\theta\quad r\leq t.
$$
Now we take $t_0\in[r,t]$ and then
\begin{align*}
\|y^\varepsilon(r,t) &-y^\varepsilon(r,t_0)\|_{X_\varepsilon^\alpha}\\ 
& = \|e^{-A_\varepsilon^+ (r-t_0)}[y(t_0,t)-v^\varepsilon(t_0)]\|_{X_\varepsilon^\alpha} \\
& +\Big\| \int_{t_0}^r e^{-A_\varepsilon^+(r-\theta)}[H_\varepsilon(y^\varepsilon(\theta,t),s_*^\varepsilon(y^\varepsilon(\theta,t)))-H_\varepsilon(y^\varepsilon(\theta,t_0),s_*^\varepsilon(y^\varepsilon(\theta,t_0)))]\, d\theta\Big\|_{X_\varepsilon^\alpha} \\
&\leq \rho M^2 e^{-\gamma (r-t_0)} \int_{t_0}^{t} e^{-(-\gamma-\rho M(1+\Delta))(\theta-t_0)}\|\xi^\varepsilon(\theta)\|_{X_\varepsilon^{\alpha}}\,d\theta\\
&+\rho M\int_{r}^{t_0} e^{-\gamma(r-\theta)}(1+\Delta)\| y^\varepsilon(\theta,t)-y^\varepsilon(\theta,t_0)\|_{X_\varepsilon^\alpha}\,d\theta. 
\end{align*}
By Gronwall's inequality
$$
\|y^\varepsilon(r,t)-y^\varepsilon(r,t_0)\|_{X_\varepsilon^\alpha} \leq \rho M^2 \int_{t_0}^{t} e^{-(-\gamma-\rho M(1+\Delta))(\theta-r)} \|\xi^\varepsilon(\theta)\|_{X_\varepsilon^\alpha}\,d\theta.
$$
Since
$$
z^\varepsilon(t)=e^{-A_\varepsilon^- (t-t_0)}z^\varepsilon(t_0)+\int_{t_0}^t e^{-A_\varepsilon^- (t-r)}G_\varepsilon(v^\varepsilon(r),z^\varepsilon(r))\,dr,
$$
we can estimate $\xi^\varepsilon(t)$ as
\begin{align*}
e^{\beta (t-t_0)} \|\xi^\varepsilon(t) & \|_{X_\varepsilon^{\alpha}} \leq M \|\xi^\varepsilon(t_0)\|_{X_\varepsilon^{\alpha}}+ \Big[\rho M+\frac{\rho^2 M^2(1+\Delta)}{\beta-\gamma-\rho M(1+\Delta) }\Big]\int_{t_0}^t e^{\beta(r-t_0)}\|\xi^\varepsilon(r)\|_{X_\varepsilon^{\alpha}} \,dr\\
& +\frac{\rho^2M^3(1+\Delta)}{\beta-\gamma-\rho M (1+\Delta)}\int_{t_0}^t e^{-(\beta-\gamma-\rho M(1+\Delta)(\theta-t_0)}e^{\beta (\theta-t_0)}\|\xi^\varepsilon(\theta)\|_{X_\varepsilon^{\alpha}}\,d\theta\\
&\leq M \|\xi^\varepsilon(t_0)\|_{X_\varepsilon^{\alpha}}+\Big[\rho M+\frac{\rho^2 M^2(1+\Delta)(1+M)}{\beta-\gamma-\rho M(1+\Delta)}\Big]\int_{t_0}^t e^{\beta(r-t_0)}\|\xi^\varepsilon(r)\|_{X_\varepsilon^{\alpha}}\,dr.
\end{align*}
By Gronwall's inequality
$$
\|\xi^\varepsilon(t)\|_{X_\varepsilon^{\alpha}}\leq M\|\xi^\varepsilon(t_0)\|_{X_\varepsilon^{\alpha}}e^{-(L-\beta)(t-t_0)},
$$
and then
$$
\|z^\varepsilon(t)-s_*^\varepsilon(v^\varepsilon(t))\|_{X_\varepsilon^{\alpha}}=\|\xi^\varepsilon(t)\|_{X_\varepsilon^{\alpha}}\leq M\|\xi^\varepsilon(t_0)\|_{X_\varepsilon^{\alpha}}e^{-(L-\beta)(t-t_0)}.
$$

Now if $u^\varepsilon:=T_\varepsilon(t)u_0^\varepsilon=v^\varepsilon(t)+z^\varepsilon(t)$, $t\in\R$, denotes the solution through at $u_0^\varepsilon=v_0^\varepsilon+z_0^\varepsilon\in \mathcal{A}_\varepsilon$, then
$$
\|z^\varepsilon(t)-s_*^\varepsilon(v^\varepsilon(t))\|_{X_\varepsilon^{\alpha}}\leq M\|z_0^\varepsilon-s_*^\varepsilon(v_0^\varepsilon)\|_{X_\varepsilon^{\alpha}}e^{-(L-\beta)(t-t_0)}.
$$
Since $\{T_\varepsilon(t)u_0^\varepsilon\,;\,t\in\R\}\subset\mathcal{A}_\varepsilon$ is bounded, letting $t_0\to-\infty$ we obtain $T_\varepsilon(t)u_0^\varepsilon=v^\varepsilon(t)+s_*^\varepsilon(v^\varepsilon(t))\in\mathcal{M}_\varepsilon$. That is $\mathcal{A}_\varepsilon\subset\mathcal{M}_\varepsilon$. Moreover, if $B_\varepsilon\subset X_\varepsilon^\alpha$ is a bounded set and  $u_0^\varepsilon=v_0^\varepsilon+z_0^\varepsilon\in B_\varepsilon$, and we conclude that $T_\varepsilon(t)u_0^\varepsilon=v^\varepsilon(t)+z^\varepsilon(t)$ satisfies
\begin{align*}
\sup_{u_0^\varepsilon\in B_\varepsilon}\inf_{w\in\mathcal{M}_\varepsilon}\|T_\varepsilon(t)u_0^\varepsilon-w\|_{X_\varepsilon^\alpha}&\leq \sup_{u_0^\varepsilon\in B_\varepsilon}\|z^\varepsilon(t)-s_*^\varepsilon(v^\varepsilon(t))\|_{X_\varepsilon^\alpha}\\
&\leq Me^{-(L-\beta)(t-t_0)} \sup_{u_0^\varepsilon\in B_\varepsilon}\|z_0^\varepsilon-s_*^\varepsilon(v_0^\varepsilon)\|_{X_\varepsilon^{\alpha}},
\end{align*}
which implies
$$
\tn{dist}_\varepsilon(T_\varepsilon(t)B_\varepsilon,\mathcal{M}_\varepsilon)\leq Ce^{-(L-\beta)(t-t_0)},
$$
and thus the proof is complete.
\end{proof}

\begin{remark}%\label{C_convergence_reaction_diffusion}
It is well known $($see \cite{Santamaria2014} and \cite{Sell2002}$)$ the $C^0$, $C^1$ and $C^{1,\theta}$ convergences of invariant manifolds. That is  
$
\|s_\ast^\varepsilon\|_{C^0}, \|s_\ast^\varepsilon\|_{C^1}, \|s_\ast^\varepsilon\|_{C^{1,\theta}}\overset{\varepsilon\to 0} \longrightarrow 0.
$
\end{remark}

\section{Rate of Convergence}\label{Rate of Convergence}
     
In this section we estimate the convergence of attractors $\mathcal{A}_\varepsilon$ of \eqref{perturbed_problem} to the attractor $\mathcal{A}_0$ of the \eqref{unperturbed_problem} by convergence of these atrractors immersed in $\R^n$. We also prove the convergence of the nonlinear semigroups.

We start defining convergence for a family of subsets in $X_\varepsilon^\alpha$. 

\begin{definition}\label{continuity_attractor_rate}
We say that a family $\{\mathcal{A}_\varepsilon\}_{\varepsilon\in (0,\varepsilon_0]}$ of subsets of $X_\varepsilon^\alpha$ converge to $\mathcal{A}_0\subset X_0$ as $\varepsilon\to 0$ if,
$$
\tn{d}_\varepsilon(\mathcal{A}_\varepsilon,E_\varepsilon \mathcal{A}_0)\overset{\varepsilon\to 0}\longrightarrow 0.
$$ 
\end{definition}

We saw in the last section that the invariant manifold $\mathcal{M}_\varepsilon$, $\varepsilon\in (0,\varepsilon_0]$, contains the attractor $\mathcal{A}_\varepsilon$ and the flow is given by 
$$
u^\varepsilon(t)=v^\varepsilon(t)+s_\ast^\varepsilon(v^\varepsilon(t)), \quad t\in\R,
$$ 
where $v^\varepsilon(t)$ satisfy the following ordinary differential equation
\begin{equation}\label{reducted_solution_edo_reaction_diffusion}
\dot{v^\varepsilon}+A_\varepsilon^+v^\varepsilon=Q_\varepsilon f(v^\varepsilon+ s_\ast^\varepsilon(v^\varepsilon)).
\end{equation}

Since $v^\varepsilon\in Y_\varepsilon$ we can consider $H_\varepsilon(v^\varepsilon)=Q_\varepsilon f(v^\varepsilon+ s_\ast^\varepsilon(v^\varepsilon))$ a continuously differentiable map in $Y_\varepsilon$. We denote $\tilde{T}_\varepsilon(t)$ the nonlinear semigroup generated by the solution of \eqref{reducted_solution_edo_reaction_diffusion},  $\tilde{T}_\varepsilon=\tilde{T}_\varepsilon(1)$ and $T_0=T_0(1)$.

Now we are ready to estimate the projected nonlinear semigroup reduced to the invariant manifold.

\begin{theorem}\label{continuity_nonlinear_semigroup}
There is a positive constant $C$ independent of $\varepsilon$ such that
\begin{equation}\label{estimate_nonlinear1}
\|\tilde{T}_\varepsilon-E_\varepsilon T_0 M_\varepsilon\|_{\LL(Y_\varepsilon,X_\varepsilon^\alpha)}\leq C (\tau(\varepsilon)+\rho(\varepsilon)) 
\end{equation}
and for each $w^\varepsilon\in \mathcal{A}_\varepsilon$ and $w^0\in \mathcal{A}_0$, 
\begin{equation}\label{estimate_nonlinear2}
\|E_\varepsilon T_0 M_\varepsilon Q_\varepsilon w^\varepsilon-E_\varepsilon T_0 w^0\|_{X_\varepsilon^\alpha}\leq C\|Q_\varepsilon w^\varepsilon-E_\varepsilon w^0\|_{X_\varepsilon^\alpha}. 
\end{equation}
\end{theorem}
\begin{proof}
For $\bar{w}^\varepsilon\in Y_\varepsilon$ we have    
\begin{multline*}
\|\tilde{T}_\varepsilon(t)\bar{w}^\varepsilon-E_\varepsilon T_0(t)M_\varepsilon \bar{w}^\varepsilon\|_{X_\varepsilon^\alpha} \leq \|(e^{-A^+_\varepsilon t}-E_\varepsilon e^{-A_0 t}M_\varepsilon)\bar{w}^\varepsilon\|_{X_\varepsilon^\alpha} \\ +\int_0^{t} \|e^{-A_\varepsilon^+(t-s)}H_\varepsilon(\tilde{T}_\varepsilon(s)\bar{w}^\varepsilon)-E_\varepsilon e^{-A_0(t-s)}f_0(T_0(s)M_\varepsilon \bar{w}^\varepsilon)\|_{X_\varepsilon^\alpha}\,ds.
\end{multline*}
But 
\begin{align*}
\int_0^{t} \|e^{-A_\varepsilon^+(t-s)}& H_\varepsilon(\tilde{T}_\varepsilon(s)\bar{w}^\varepsilon)-E_\varepsilon e^{-A_0(t-s)}f_0(T_0(s)M_\varepsilon \bar{w}^\varepsilon)\|_{X_\varepsilon^\alpha}\,ds \\
&\leq \int_0^{t} \|e^{-A_\varepsilon^+(t-s)}H_\varepsilon(\tilde{T}_\varepsilon(s)\bar{w}^\varepsilon)-E_\varepsilon e^{-A_0(t-s)}M_\varepsilon H_\varepsilon(\tilde{T}_\varepsilon(s)\bar{w}^\varepsilon)\|_{X_\varepsilon^\alpha}\,ds \\
& + \int_0^{t} \|E_\varepsilon e^{-A_0(t-s)}M_\varepsilon [H_\varepsilon(\tilde{T}_\varepsilon(s)\bar{w}^\varepsilon)-H_\varepsilon(E_\varepsilon T_0(s)M_\varepsilon\bar{w}^\varepsilon)]\|_{X_\varepsilon^\alpha}\,ds\\ 
& + \int_0^{t} \|E_\varepsilon e^{-A_0(t-s)}[M_\varepsilon H_\varepsilon(E_\varepsilon T_0(s)M_\varepsilon \bar{w}^\varepsilon)-E_\varepsilon f_0(T_0(s)M_\varepsilon\bar{w}^\varepsilon)]\|_{X_\varepsilon^\alpha}\,ds.
\end{align*}
If we denote the three last integrals by $I_1$, $I_2$ and $I_3$ respectively and consider $0\leq t\leq 1$, we have from \eqref{rate_nonlinearity} and the Lemmas \ref{projectiom_estimate} and \ref{Linear_estimate} that
$$
I_1\leq C \tau(\varepsilon),\quad I_2\leq C\int_0^t \|\tilde{T}_\varepsilon(s)\bar{w}^\varepsilon-E_\varepsilon T_0(s)M_\varepsilon\bar{w}^\varepsilon\|_{X_\varepsilon^\alpha}\,ds
$$
and  
\begin{align*}
I_3 &\leq C\int_0^{t} \|M_\varepsilon H_\varepsilon(E_\varepsilon T_0(s)M_\varepsilon \bar{w}^\varepsilon)-f_0(T_0(s)M_\varepsilon\bar{w}^\varepsilon)\|_{X_0}\,ds\\
&\leq C\int_0^{t} \|H_\varepsilon(E_\varepsilon T_0(s)M_\varepsilon \bar{w}^\varepsilon)- E_\varepsilon f_0(T_0(s)M_\varepsilon\bar{w}^\varepsilon)\|_{X_\varepsilon^\alpha}\,ds\\
&\leq C\int_0^{t} \|Q_\varepsilon f_\varepsilon(E_\varepsilon T_0(s)M_\varepsilon \bar{w}^\varepsilon+s_\ast^\varepsilon(E_\varepsilon T_0(s)M_\varepsilon \bar{w}^\varepsilon))- E_\varepsilon f_0(T_0(s)M_\varepsilon\bar{w}^\varepsilon)\|_{X_\varepsilon^\alpha}\,ds\\
&\leq C\int_0^{t} \|Q_\varepsilon f_\varepsilon(E_\varepsilon T_0(s)M_\varepsilon \bar{w}^\varepsilon+s_\ast^\varepsilon(E_\varepsilon T_0(s)M_\varepsilon \bar{w}^\varepsilon))-Q_\varepsilon E_\varepsilon f_0(T_0(s)M_\varepsilon\bar{w}^\varepsilon)\|_{X_\varepsilon^\alpha}\,ds\\
&+ C\int_0^{t} \|Q_\varepsilon E_\varepsilon f_0(T_0(s)M_\varepsilon\bar{w}^\varepsilon)- E_\varepsilon f_0(T_0(s)M_\varepsilon\bar{w}^\varepsilon)\|_{X_\varepsilon^\alpha}\,ds\\
&\leq C(|\!|\!|s_\ast^\varepsilon |\!|\!|+\rho(\varepsilon))+ C\tau(\varepsilon).
\end{align*}
Hence
$$
\|\tilde{T}_\varepsilon(t)\bar{w}^\varepsilon-E_\varepsilon T_0(t)M_\varepsilon\bar{w}^\varepsilon\|_{X_\varepsilon^\alpha}\leq C(\tau(\varepsilon)+\rho(\varepsilon))+C\int_0^t\|\tilde{T}_\varepsilon(s)\bar{w}^\varepsilon-E_\varepsilon T_0(s)M_\varepsilon\bar{w}^\varepsilon\|_{X_\varepsilon^\alpha}\,ds.
$$
By Gronwall's inequality and then taking $t=1$ we obtain \eqref{estimate_nonlinear1}.

Now, to prove \eqref{estimate_nonlinear2}, note that from variation of constants formula and the Gronwall's inequality, we have $T_0(\cdot)$ a Lipschitz map in $X_0$, thus
\begin{align*}
\|E_\varepsilon T_0 M_\varepsilon Q_\varepsilon w^\varepsilon-E_\varepsilon T_0 w^0\|_{X_\varepsilon^\alpha}&\leq C\|T_0 M_\varepsilon Q_\varepsilon w^\varepsilon- T_0 w^0\|_{X_0}\\
&\leq C\|M_\varepsilon Q_\varepsilon w^\varepsilon-w^0\|_{X_0}\\
&=C\|M_\varepsilon Q_\varepsilon w^\varepsilon-M_\varepsilon E_\varepsilon w^0\|_{X_0}\\
&\leq C\|Q_\varepsilon w^\varepsilon-E_\varepsilon w^0\|_{X_\varepsilon^\alpha}.
\end{align*}
\cqd
\end{proof}

As a consequence of Theorem \ref{continuity_nonlinear_semigroup} we have the following result.

\begin{theorem}\label{estimate_attractor} 
Let $\mathcal{A}_\varepsilon$ be the attractor for \eqref{perturbed_problem} and $\mathcal{A}_0$ the attractor of the \eqref{unperturbed_problem}. Then there is a positive constant $C$ independent of $\varepsilon$ such that
$$
\tn{d}_\varepsilon(\mathcal{A}_\varepsilon,E_\varepsilon\mathcal{A}_0)\leq C(\tau(\varepsilon)+\rho(\varepsilon))+\tn{d}_\varepsilon(Q_\varepsilon\mathcal{A}_\varepsilon,E_\varepsilon \mathcal{A}_0).
$$
\end{theorem}
\begin{proof} Let $u^\varepsilon\in \mathcal{A}_\varepsilon$, and $u^0\in \mathcal{A}_0$. We can write $u^0=T_0z^0$, $z^0\in\mathcal{A}_0$ and $u^\varepsilon=v^\varepsilon+s_\ast^\varepsilon(v^\varepsilon)$, with $v^\varepsilon=\tilde{T}_\varepsilon z^\varepsilon$, $z^\varepsilon\in Q_\varepsilon\mathcal{A}_\varepsilon$. Thus, by \eqref{estimate_nonlinear1} and \eqref{estimate_nonlinear2}, we have
\begin{align*}
\|u^\varepsilon-E_\varepsilon u^0\|_{X_\varepsilon^\alpha} &=\|v^\varepsilon+s_\ast^\varepsilon(v^\varepsilon)- E_\varepsilon T_0 z^0\|_{X_\varepsilon^\alpha} \\
& \leq \|\tilde{T}_\varepsilon z^\varepsilon+s_\ast^\varepsilon(\tilde{T}_\varepsilon z^\varepsilon)-E_\varepsilon T_0 z^0\|_{X_\varepsilon^\alpha}\\
& \leq \|\tilde{T}_\varepsilon z^\varepsilon-E_\varepsilon T_0 M_\varepsilon z^\varepsilon\|_{X_\varepsilon^\alpha}+|\!|\!|s_\ast^\varepsilon |\!|\!| +\|E_\varepsilon T_0 M_\varepsilon z^\varepsilon-E_\varepsilon T_0 z^0\|_{X_\varepsilon^\alpha}\\
& \leq C(\tau(\varepsilon)+\rho(\varepsilon)) + C\|z^\varepsilon-E_\varepsilon z^0\|_{X_\varepsilon^\alpha}.
\end{align*}
\cqd
\end{proof}

Recall that we denoted $\{\varphi_i^0\}_{i=1}^n$ the eigenfunctions associated to the $n$ eigenvalues of $A_0$. Since $Y_\varepsilon$ and $X_0$ are $n$ dimensional, we can consider
$$
Y_\varepsilon=\tn{span}[Q_\varepsilon E_\varepsilon\varphi_1^0,\cdots,Q_\varepsilon E_\varepsilon\varphi_n^0]\quad\tn{and}\quad X_0=\tn{span}[\varphi_1^0,\cdots,\varphi_n^0]
$$ 
and define the isomorphisms $j_\varepsilon:Y_\varepsilon\to \R^n$ and $j_0:X_0\to \R^n$ by 
$$
\sum_{j=1}^n z_j^\varepsilon Q_\varepsilon E_\varepsilon \varphi_j^0\overset{j_\varepsilon}\longrightarrow (z_1^\varepsilon,\cdots,z_n^\varepsilon)\quad\tn{and}\quad \sum_{j=1}^n z_j^0\varphi_j^0\overset{j_0}\longrightarrow (z_1^0,\cdots,z_n^0),
$$
where in $\R^n$ we consider the following norm
\begin{equation}\label{norm_R_n}
\|z\|_{\R^n}=\Big(\sum_{j=1}^n z_i^2(\lambda_i^\varepsilon)^{2\alpha}\Big)^{\frac{1}{2}},\quad z=(z_1,\cdots,z_n)\in \R^n,
\end{equation}
where $\{\lambda_i^\varepsilon\}_{j=1}^n$ is the first $n$ eigenvalues of $A_\varepsilon$.

\begin{lemma}\label{estimate_isomorphism}
For $\bar{w}^\varepsilon\in Y_\varepsilon$ and $w^0\in X_0$ we have the following inequality
$$
\|j_\varepsilon \bar{w}^\varepsilon-j_0w^0\|_{\R^n}\leq C(\|\bar{w}^\varepsilon-E_\varepsilon w^0\|_{X_\varepsilon^\alpha}+\tau(\varepsilon)),
$$
where $C$ is a constant independent of $\varepsilon$.
\end{lemma}
\begin{proof}
In fact,
\begin{align*}
\bar{w}^\varepsilon-E_\varepsilon w^0 &=\sum_{j=1}^n z_j^\varepsilon Q_\varepsilon E_\varepsilon\varphi_j^0-E_\varepsilon \sum_{j=1}^n z_j^0 \varphi_j^0 \\
&=(Q_\varepsilon E_\varepsilon-E_\varepsilon)\sum_{j=1}^n z_j^\varepsilon\varphi_j^0+E_\varepsilon\sum_{j=1}^n (z_j^\varepsilon-z_j^0)\varphi_j^0,
\end{align*}
since $M_\varepsilon E_\varepsilon=I_{X_0}$, we obtain
$$
\sum_{j=1}^n (z_j^\varepsilon-z_j^0)\varphi_j^0=M_\varepsilon(\bar{w}^\varepsilon-E_\varepsilon w^0)-M_\varepsilon (Q_\varepsilon E_\varepsilon-E_\varepsilon)\sum_{j=1}^n z_j^\varepsilon\varphi_j^0,
$$
which implies
$$
\|j_\varepsilon \bar{w}^\varepsilon-j_0w^0\|_{\R^n}\leq C(\|\bar{w}^\varepsilon-E_\varepsilon w^0\|_{X_\varepsilon^\alpha}+\|Q_\varepsilon E_\varepsilon-E_\varepsilon\|_{\LL(X_0,X_\varepsilon^\alpha)})
$$
The result follows by Lemma \ref{spectral_projection}.
\cqd
\end{proof}

We now restrict our attention to dynamical systems $\tilde{T}_\varepsilon$ and $T_0$ that acts in $Y_\varepsilon$ and $X_0$ respectively. Thus we consider the systems generated by following ordinary differential equations,
$$
\begin{cases}
\dot{v}^\varepsilon(t)+A_\varepsilon^+ v^\varepsilon(t)=H_\varepsilon(v^\varepsilon(t)),\quad t\in\R,\\ v^\varepsilon(0)=\bar{w}^\varepsilon\in Y_\varepsilon,\quad\varepsilon\in(0,\varepsilon_0],
\end{cases}
$$
$$
\begin{cases}
\dot{u}^0(t)+A_0 u^0(t)=f_0(u^0(t)),\quad t\in\R,\\u^0(0)=M_\varepsilon\bar{w}^\varepsilon\in X_0.
\end{cases}
$$
Let $z_0^0,z_0^\varepsilon\in\R^n$ such that $\bar{w}^\varepsilon=j_\varepsilon^{-1}z_0^\varepsilon$ and $M_\varepsilon\bar{w}^\varepsilon=j_0^{-1}z_0^0$, then $j_\varepsilon(v^\varepsilon(t))$ and $j_0(u^0(t))$ satisfy the following equation in $\R^n$, 
\begin{equation}\label{ODE_1}
\begin{cases}
\dot{z}^\varepsilon(t)+j_\varepsilon A_\varepsilon^+j_\varepsilon^{-1} z^\varepsilon(t)=j_\varepsilon H_\varepsilon(j_\varepsilon^{-1} z^\varepsilon(t)),\quad t\in\R,\\ z^\varepsilon(0)=z_0^\varepsilon,\quad\varepsilon\in(0,\varepsilon_0],
\end{cases}
\end{equation}
\begin{equation}\label{ODE_2}
\begin{cases}
\dot{z}^0(t)+j_0 A_0 j_0^{-1}z^0(t)=j_0f_0(j_0^{-1}z^0(t)),\quad t\in\R,\\z^0(0)=z_0^0.
\end{cases}
\end{equation}
Since we assume that the limiting problem \eqref{unperturbed_problem} generates a Morse-Smale semigroup in $X_0$,  the perturbed problem \eqref{perturbed_problem} generate a Morse-Smale semigroup in $X_\varepsilon^\alpha$,
thus if we denote $S_0$ and $S_\varepsilon$ the time one map of the systems \eqref{ODE_1} and \eqref{ODE_2},  that is, let $S_\varepsilon(\cdot)$ and $S_0(\cdot)$ be the nonlinear semigroup in $\R^n$ given by the solutions of \eqref{ODE_1} and \eqref{ODE_2} respectively, define $S_\varepsilon=S_\varepsilon(1)$ and $S_0=S_0(1)$. We have $S_0$ and $S_\varepsilon$ Morse-Smale semigroups in $\R^n$ and we denote $\bar{\mathcal{A}}_0$ and $\bar{\mathcal{A}}_\varepsilon$ its attractors respectively.

Next we estimate the convergence of the nonlinear semigroups acting in $\R^n$.

\begin{proposition}\label{nonlinear_estimate_1}
There is a constant $C$ independent of $\varepsilon$ such that 
$$
\|S_\varepsilon-S_0\|_{L^\infty(\R^n)}\leq C(\tau(\varepsilon)+\rho(\varepsilon)).
$$
\end{proposition}
\begin{proof} We have
\begin{align*}
\|S_\varepsilon-S_0\|_{L^\infty(\R^n)}&=\sup_{z\in\R^n} \|z^\varepsilon(1)-z^0(1)\|_{\R^n}=\sup_{z\in\R^n}\|j_\varepsilon(v^\varepsilon(1))-j_0(u^0(1))\|_{\R^n},
\end{align*}
where $z^0(0)=z=z^\varepsilon(0)$, that is $j_0^{-1}(z)=u^0(0)=M_\varepsilon \bar{w}^\varepsilon$ and $j_\varepsilon^{-1}(z)=v^\varepsilon(0)=\bar{w}^\varepsilon$. By  Lemma \ref{estimate_isomorphism}, we have
\begin{align*}
\|j_\varepsilon(v^\varepsilon(1))-j_0(u^0(1))\|_{\R^n}&\leq C(\|v^\varepsilon(1)-E_\varepsilon u^0(1)\|_{X_\varepsilon^\alpha}+\tau(\varepsilon))\\
&=C(\|\tilde{T}_\varepsilon w^\varepsilon-E_\varepsilon T_0 M_\varepsilon\bar{w}^\varepsilon\|_{X_\varepsilon^\alpha}+\tau(\varepsilon)).
\end{align*}
The result now follows by Theorem \ref{continuity_nonlinear_semigroup}.
\cqd
\end{proof}

The next result relates the projected attractors with the attractors in $\R^n$.

\begin{theorem}\label{estimate_attractor_projected}
There is a constant $C$ independent of $\varepsilon$ such that
$$
\tn{d}_\varepsilon(Q_\varepsilon\mathcal{A}_\varepsilon,E_\varepsilon \mathcal{A}_0)\leq C(\tn{dist}_{\R^n}(\bar{\mathcal{A}}_\varepsilon,\bar{\mathcal{A}}_0)+\tau(\varepsilon)),
$$
where
$\tn{dist}_{\R^n}(\bar{\mathcal{A}}_\varepsilon,\bar{\mathcal{A}}_0)=\sup_{u\in \bar{\mathcal{A}}_\varepsilon}\inf_{v\in \bar{\mathcal{A}}_0}\|u-v\|_{\R^n}$ denotes the Hausdorff semidistance in $\R^n$ with respect to \eqref{norm_R_n}.
\end{theorem}
\begin{proof}
Let $w^0\in\mathcal{A}_0$ and $w^\varepsilon\in \mathcal{A}_\varepsilon$, we can write
$$
w^0=j_0^{-1}(z^0)=\sum_{j=1}^n z_j^0\varphi_j^0,\quad\tn{and}\quad Q_\varepsilon w^\varepsilon=j_\varepsilon^{-1}(z^\varepsilon)=\sum_{j=1}^n z_j^\varepsilon\psi_j^\varepsilon,
$$ 
where $\psi_j^\varepsilon=Q_\varepsilon E_\varepsilon \varphi_j^0$, $j=1,\dots,n$, $z^0\in \bar{\mathcal{A}}_0$ and $z^\varepsilon\in \bar{\mathcal{A}}_\varepsilon$. Thus
\begin{align*}
\|Q_\varepsilon w^\varepsilon-E_\varepsilon w^0\|_{X_\varepsilon^\alpha}&=\|j_\varepsilon^{-1}(z^\varepsilon)-E_\varepsilon j_0^{-1}(z^0)\|_{X_\varepsilon^\alpha}=\|\sum_{j=1}^n z_j^\varepsilon\psi_j^\varepsilon-E_\varepsilon \sum_{j=1}^n z_j^0\varphi_j^0\|_{X_\varepsilon^\alpha}\\
&\leq \|\sum_{j=1}^n z_j^\varepsilon(\psi_j^\varepsilon-E_\varepsilon\varphi_j^0)\|_{X_\varepsilon^\alpha}+\|\sum_{j=1}^n (z_j^\varepsilon-z_j^0)E_\varepsilon\varphi_j^0\|_{X_\varepsilon^\alpha}.
\end{align*}
But 
\begin{align*}
\|\sum_{j=1}^n z_j^\varepsilon(\psi_j^\varepsilon-E_\varepsilon\varphi_j^0)\|_{X_\varepsilon^\alpha}=&\|\sum_{j=1}^n z_j^\varepsilon(Q_\varepsilon E_\varepsilon \varphi_j^0-E_\varepsilon\varphi_j^0)\|_{X_\varepsilon^\alpha}\\
&\leq \sup_{1\leq j \leq n}|z_j^\varepsilon| \|(Q_\varepsilon E_\varepsilon -E_\varepsilon)\sum_{j=1}^n\varphi_j^0\|_{X_\varepsilon^\alpha}\\
&\leq C\|Q_\varepsilon E_\varepsilon-E_\varepsilon\|_{\LL(X_0,X_\varepsilon^\alpha)}
\end{align*}
and
\begin{align*}
\|\sum_{j=1}^n (z_j^\varepsilon-z_j^0)E_\varepsilon\varphi_j^0\|_{X_\varepsilon^\alpha}^2&=\sum_{j=1}^n(\lambda_j^\varepsilon)^{2\alpha}\|Q_j^\varepsilon\sum_{i=1}^n(z_i^\varepsilon-z_i^0)E_\varepsilon\varphi_i^0\|^2_{X_\varepsilon}\\
&\leq \sum_{j=1}^n(\lambda_j^\varepsilon)^{2\alpha}(z_j^\varepsilon-z_j^0)^2\|Q_j^\varepsilon E_\varepsilon\varphi_j^0\|^2_{X_\varepsilon}\\
&\leq C \sum_{j=1}^n (\lambda_j^\varepsilon)^{2\alpha}(z_j^\varepsilon-z_j^0)^2\\
&=C\|z^\varepsilon-z^0\|^2_{\R^n}.
\end{align*}
By Lemma \ref{projectiom_estimate} the result follows .  
\cqd
\end{proof} 

It follows from Theorems \ref{estimate_attractor} and \ref{estimate_attractor_projected} the following estimate
\begin{equation}\label{estimate_attractor_projected2}
\tn{d}_\varepsilon(\mathcal{A}_\varepsilon,E_\varepsilon\mathcal{A}_0)\leq C(\tau(\varepsilon)+\rho(\varepsilon)+\tn{dist}_{\R^n}(\bar{\mathcal{A}}_\varepsilon,\bar{\mathcal{A}}_0)).
\end{equation}
Therefore we need to estimate  $\tn{dist}_{\R^n}(\bar{\mathcal{A}}_\varepsilon,\bar{\mathcal{A}}_0)$. To this end we use a result proved in \cite{Santamaria2014}, where an aplication of the Shadowing Theory in dynamical system was used to obtain a rate of convergence for attractors for semigroups in finite dimension.

\section{Shadowing Theory and Rate of Convergence}\label{Shadowing Theory and Rate of Convergence}

In this section we prove the main result of this paper but before, we make a brief overview of the some results presented in \cite{Santamaria2014}. 

Let $T:\R^n\to \R^n$ be a continuous function. Recall that the discrete dynamical system generated by $T$ is defined by $T^0=I_{\R^n}$ and, for $k\in\N$, $T^k=T\circ\dots\circ T$ is the $k$th iterate of $T$. The notions of Morse-Smale systems, hyperbolic fixed points, stable and unstable manifolds for a function $T$ are similar to the continuous case (see \cite{Hale1988}).  

\begin{definition}
A trajectory (or global solution) of the discrete dynamical system generated by $T$ is a sequence $\{x_n\}_{n\in\Z}\subset \R^n$, such that, $x_{n+1}=T(x_n)$, for all $n\in\Z$. 
\end{definition}

\begin{definition}
We say that a sequence $\{x_n\}_{n\in\Z}$ is a $\delta$-pseudo-trajectory of $T$ if 
$$
\|Tx_x-x_{k+1}\|_{\R^n}\leq\delta,\quad\tn{for all }n\in\Z. 
$$
\end{definition}

\begin{definition}
We say that a point $x\in \R^n$ $\varepsilon$-shadows a $\delta$-pseudo-trajectory $\{x_k\}$ on $U\subset \R^n$ if the inequality
$$
\|T^kx-x_k\|_{\R^n}\leq \varepsilon,\quad k\in\Z,
$$
hold.
\end{definition}

\begin{definition}
The map $T$ has the Lipschitz Shadowing Property (LpSP) on $U\subset \R^n$, if there are constants $L, \delta_0>0$ such that for any $0< \delta\leq \delta_0 $, any $\delta$-pseudo-trajectory of $T$ in $U$ is $(L\delta)$-shadowed by a trajectory of $T$ in $\R^n$, that is, for any sequence $\{x_k\}_k\subset U\subset \R^n$ with
$$
\|Tx_k-x_{k+1}\|_{\R^n}\leq \delta\leq \delta_0,\quad k\in\Z,
$$
there is a point $x\in X$ such that the inequality
$$
\|T^kx-x_k\|_{\R^n}\leq L\delta,\quad k\in\Z,
$$
hold.
\end{definition}

Let $T:\R^m\to \R^m$ be a Morse Smale function which has an attractor $\mathcal{A}$. Since $\mathcal{A}$ is compact and has all dynamic of the system, we can restrict our attention on a neigborhood  $\mathcal{N}(\mathcal{A})$ of $\mathcal{A}$, thus we consider the space $C^1(\mathcal{N}(\mathcal{A}),\R^m)$ with the $C^1$-topology.

The next result can be found in \cite{Santamaria2014}. It describes an application of Shadowing Theory to rate of convergence of attractors.

\begin{proposition}\label{prop_estimate_LPSP}
Let $T_1, T_2:X\to X$ be maps which has a global attractors $\mathcal{A}_1,\mathcal{A}_2$. Assume that $\mathcal{A}_1,\mathcal{A}_2\subset \mathcal{U}\subset X$, that $T_1,T_2$ have both the LpSP on $\mathcal{U},$ with parameters $L,\delta_0$ and $\|T_1-T_2\|_{\mathcal{L}^\infty(\mathcal{U},X)}\leq \delta$. Then we have 
$$
\tn{dist}_H(\mathcal{A}_1,\mathcal{A}_2)\leq \|T_1-T_2\|_{\LL^\infty(\mathcal{U},X)}.
$$
\end{proposition}
\begin{proof}
Since $T_1$ and $T_2$ has the LpSP on $\mathcal{U}$. Take a trajectory $\{y_n\}_n$ of $T_2$ in $ \mathcal{A}_2$, then $\{y_n\}_n$ is a $\delta$-pesudo-trajectory of $T_1$ with $\delta=\|T_1-T_2\|_{\LL^\infty(\mathcal{U},X)}\leq \delta_0$. By LpSP there is a trajectory $\{x_n\}_n\subset X$ of $T_1$ such that $\|x_n-y_n\|_X\leq L\delta$ for al $n\in\Z$, hence $\{x_n\}_n\subset \mathcal{A}_1$. Since $\{y_n\}_n$ is arbitrary the result follows. 
\cqd
\end{proof}

As a consequence of the Proposition \ref{prop_estimate_LPSP}, we have the following result.

\begin{theorem}\label{Shadowing_reaction_diffusion}[Arrieta and Santamar\' ia]
Let $T:\R^m\to \R^m$ be a Morse-Smale function with a global attractor $\mathcal{A}$. Then there are a positive constant $L$, a neighborhood $\mathcal{N}(\mathcal{A})$ of $\mathcal{A}$ and a neighborhood $\mathcal{N}(T)$ of $T$ in the $C^1(\mathcal{N}(\mathcal{A}),\R^m)$ topology such that, for any $T_1,T_2\in \mathcal{N}(T)$ with attractors  $\mathcal{A}_1$, $\mathcal{A}_2$ respectively, we have
$$
\tn{dist}_H(\mathcal{A}_1,\mathcal{A}_2)\leq L\|T_1-T_2\|_{L^\infty(\mathcal{N}(\mathcal{A}),\R^m)}.
$$
\end{theorem}
\begin{proof}
It is proved in \cite{Pilyugun1999} that a structurally stable dynamical system on a compact manifold has the LpSP and follows by \cite{Bortolan} that a Morse-Smale system is structurally stable. Putting this results together the authors in \cite{Santamaria2014} proved that a discrete Morse-Smale semigroup $T$ has the LpSP in a neighborhood $\mathcal{N}(\mathcal{A})$ of its attractor $\mathcal{A}$. The result now follows by Proposition \ref{prop_estimate_LPSP}.
\cqd
\end{proof}

Now we return to the end of the previous section. Since $S_0$ is a Morse-Smale semigroup in $\R^n$ and $\|S_\varepsilon-S_0\|_{L^\infty(\R^n)}\to 0$ as $\varepsilon\to 0$. By Theorem \ref{Shadowing_reaction_diffusion}, we have
\begin{equation}\label{estimate_attractor_1}
\tn{dist}_{\R^n}(\bar{\mathcal{A}}_\varepsilon,\bar{\mathcal{A}}_0)\leq C\|S_\varepsilon-S_0\|_{L^\infty(\R^n,\R^n)},
\end{equation} 
where $C$ is a constant independent of $\varepsilon$.

We are ready to prove the main result of this paper.

\begin{theorem}\label{main_result_theorem_estimate}
Let $\mathcal{A}_\varepsilon$ be the attractor for \eqref{perturbed_problem} and $\mathcal{A}_0$ the attractor of the \eqref{unperturbed_problem}. Then there is a positive constant $C$ independent of $\varepsilon$ such that
$$
\tn{d}_\varepsilon(\mathcal{A}_\varepsilon,E_\varepsilon\mathcal{A}_0)\leq C(\tau(\varepsilon)+\rho(\varepsilon)).
$$
\end{theorem}
\begin{proof} The proof follows from estimate \eqref{estimate_attractor_projected2}, \eqref{estimate_attractor_1} and Proposition \ref{nonlinear_estimate_1}.
\cqd
\end{proof}

\section{Further Comments}\label{Further Comments}

In this section we consider a more general class of linear operators than that seen in the Section \ref{Functional Setting}. We will observe that for this class of operators, that was considered in \cite{Carvalho2010a}, the theory developed in the previous section can be applied with some additional a priori estimates.

\begin{definition}
We say that a family of linear operators $\{A_\varepsilon\}_{\varepsilon\in (0,\varepsilon_0]}$ is of class $G(X_\varepsilon,C,\omega)$ if each operator $A_\varepsilon:D(A_\varepsilon)\subset X_\varepsilon\to X_\varepsilon$ generates a strongly continuous semigroup of linear operator $\{e^{-A_\varepsilon t}:t\geq 0\}\subset \LL(X_\varepsilon)$ and 
$$
\|e^{-A_\varepsilon t}\|_{\LL(X_\varepsilon,X_\varepsilon^\alpha)}\leq Ce^{\omega t},\quad\forall\,t\geq 0,
$$
for some constant $C\geq 1$ and $\omega\in \R$ independent of $\varepsilon$.
\end{definition}

\begin{definition}
We say that $\{A_\varepsilon\}_{\varepsilon\in (0,\varepsilon_0]}$ is of class $H(X_\varepsilon,M,\theta)$ if $\rho(-A_\varepsilon)$ contains the same sector
$$
\Sigma_{\theta}=\{\lambda\in \mathbb{C}:|\tn{arg}(\lambda)|\leq \frac{\pi}{2}+\theta \}
$$
and moreover
$$
\|(\lambda+A_\varepsilon)^{-1}\|_{\LL(X_\varepsilon,X_\varepsilon^\alpha)}\leq \frac{M}{|\lambda|+1},\quad\forall\,\lambda\in \Sigma_{\theta},\,\,\,\varepsilon\in (0,\varepsilon_0],
$$
where $\theta\in (0,\pi/2)$ and $M\geq 0$ do not dependent on $\varepsilon$.
\end{definition}

We have $H(X_\varepsilon,M,\theta)\subset G(X_\varepsilon,C,\omega)$. If we consider the problems \eqref{unperturbed_problem} and \eqref{perturbed_problem} with the family $\{A_\varepsilon\}_{\varepsilon\in (0,\varepsilon_0]}\in H(X_\varepsilon,M,\theta)$ and if we assume that we can obtain the linear estimates in the Lemma \ref{Linear_estimate}, then  the Theorems \ref{estimate_invariant_manifold} and \ref{main_result_theorem_estimate} also are valid in this context. 

\section{Applications to Spatial Homogenization}\label{Applications to Spatial Homogenization}

Using the approach developed in the previous sections we consider a reaction diffusion equation with large diffusion in all parts of the domain. We start doing a analysis of the large diffusion effects in order to obtain a limiting ordinary differential equation. Here we consider the convergences in one appropriated energy space and we show that the attractors of perturbed problems are closed to the attractor for  the limiting problem. Our main reference is the paper \cite{Rodriguez-Bernal2005}.

We consider the parabolic problem
\begin{equation}\label{perturbed_problem1}
\begin{cases}
u^\varepsilon_t-\tn{div}(p_\varepsilon(x) \nabla u^\varepsilon)+(\lambda+V_\varepsilon(x))u^\varepsilon = f(u^\varepsilon),\quad x\in \Omega,\,\,t>0, \\
\dfrac{\partial u^\varepsilon}{\partial \vec{n}}=0,\quad x\in \partial \Omega,\\
u^\varepsilon(0)=u^\varepsilon_0,
\end{cases}
\end{equation}
where $0<\varepsilon\leq\varepsilon_0$, $\Omega\subset \R^n$ is a bounded smooth open connected set, $\partial \Omega$ is the boundary of $\Omega$ and $\frac{\partial u^\varepsilon}{\partial \vec{n}}$ is the co-normal derivative operator with $\vec{n}$ the unit outward normal vector to $\partial \Omega$. We assume the potentials $V_\varepsilon\in L^p(\Omega)$ with  
$$
p\begin{cases} \geq 1,\quad n=1,\\\geq 2,\quad n\geq2,
\end{cases}
$$
and $V_\varepsilon$ converges to a constant $V_0 \in \R$ in $L^p(\Omega)$, that is, we consider $\tau(\varepsilon)$ an increasing positive function of $\varepsilon$ such that $\tau(0)=0$ and
\begin{equation}\label{map_tau}
\|V_\varepsilon-V_0\|_{L^p(\Omega)}\leq \tau(\varepsilon)\overset{\varepsilon\to 0}\longrightarrow 0.
\end{equation}
Note that \eqref{map_tau} implies that the spatial average of $V_\varepsilon$ converges to $V_0$ as $\varepsilon\to 0$.  We choice $\lambda\in \R$ sufficiently large for that $\tn{ess}\inf_{x\in\Omega}V_\varepsilon(x)+\lambda\geq m_0$ for some positive constant $m_0$. Moreover we will assume the diffusion is large in $\Omega$, that is, for each $\varepsilon\in (0,\varepsilon_0]$ the map $p_\varepsilon$ is positive smooth defined in $\bar{\Omega}$ satisfying
$$
p(\varepsilon):=\min_{x\in \bar{\Omega}} \{p_\varepsilon(x)\}\overset{\varepsilon\to 0}\longrightarrow \infty\quad\tn{with}\quad 0<m_0\leq p_\varepsilon(x),\quad\forall\,x\in\Omega.
$$ 

Since large diffusivity implies fast homogenization, we expect, for small values of $\varepsilon$, that the solution of this problem converge to a spatially constant function in $\Omega$. Indeed by taking the average on $\Omega$, the limiting problem as $\varepsilon$ goes to zero is given by a scalar ordinary differential equation
\begin{equation}\label{limit_problem1}
\begin{cases}
\dot{u}^0+(\lambda+V_0)u^0=f(u^0),\quad t>0,\\
u^0(0)=u^0_0,
\end{cases}
\end{equation}
which \cite{Rodriguez-Bernal2005} proves to determine the asymptotic behavior.

In this section we are concerning in how fast the dynamics of the problem \eqref{perturbed_problem1} approaches the dynamics of the problem. Following the theory developed in the last chapters, we estimate this convergence by functions $\tau(\varepsilon)$ and $p(\varepsilon)$.

Since we have established the limit problem we need to study the well posedness of \eqref{perturbed_problem1} and \eqref{limit_problem1} as abstract parabolic equation in appropriated Banach spaces. To this end, we define the operator $A_\varepsilon:\mathcal{D}(A_\varepsilon)\subset L^2(\Omega)\to L^2(\Omega)$ by
$$
\mathcal{D}(A_\varepsilon)=\{u\in H^2(\Omega): \dfrac{\partial u^\varepsilon}{\partial \vec{n}}=0\},\quad
A_\varepsilon u=-\tn{div}(p_\varepsilon \nabla u)+(\lambda+V_\varepsilon )u.
$$
We denote $L^2_{\Omega}=\{u\in H^1(\Omega): \nabla u=0 \tn{ in } \Omega\}$ and we define the operator $A_0:L^2_{\Omega}\subset L^2(\Omega)\to L^2(\Omega)$ by
$A_0 u=(\lambda+V_0)u.$ Note that $L^2_\Omega$ is the set of all almost everywhere constant function in $\Omega$.  

It is well known that $A_\varepsilon$ is a positive invertible operator with compact resolvent for each $\varepsilon\in [0,\varepsilon_0]$, hence we define in the usual way (see\cite{Henry1980}), the fractional power space $X_\varepsilon^\frac{1}{2}=H^1(\Omega)$, $\varepsilon\in (0,\varepsilon_0]$, and $X_0^\frac{1}{2}=L^2_{\Omega}$ with the scalar products 
$$
\pin{u,v}_{X_\varepsilon^{\frac{1}{2}}}=\int_\Omega p_\varepsilon \nabla u \nabla v\,dx+\int_\Omega(\lambda +V_\varepsilon)uv\,dx,\quad u,v\in X_\varepsilon^\frac{1}{2},\,\,\,\varepsilon\in (0,\varepsilon_0];
$$
$$
\pin{u,v}_{X_0^{\frac{1}{2}}}=|\Omega|^{-1}(\lambda+V_0)uv,\quad u,v\in X_0^\frac{1}{2}.
$$
The space $X_0^\frac{1}{2}$ is a one dimensional closed subspace of $X_\varepsilon^\frac{1}{2}$, $\varepsilon\in (0,\varepsilon_0]$ and $X_\varepsilon^\frac{1}{2}\subset H^1(\Omega)$ with injection constant independent of $\varepsilon$, but the injection $H^1(\Omega)\subset X_\varepsilon^\frac{1}{2}$ is not uniform, in fact is valid
$$
m_0\|u\|_{H^1}^2\leq \|u\|_{X_\varepsilon^\frac{1}{2}}^2\leq M(\varepsilon) \|u\|^2_{H^1}, 
$$
with $M(\varepsilon)\to \infty$ as $\varepsilon\to 0$ and we will show in the Corollary \ref{non_equivalence_norm1} that there is no positive constant $C$ independent of $\varepsilon$ such that 
$
\|u\|_{X_\varepsilon^\frac{1}{2}}^2\leq C\|u\|^2_{H^1}.
$
Therefore bounds for solutions in the Sobolev spaces does not give suitable estimates in the fractional power space, even though, we  consider $X_\varepsilon^\frac{1}{2}$ as phase space. 

If we denote the Nemitskii functional of $f$ by the same notation $f$, then \eqref{perturbed_problem1} and \eqref{limit_problem1} can be written as  
\begin{equation}\label{semilinear_problem1}
\begin{cases}
u^\varepsilon_t+A_\varepsilon u^\varepsilon = f(u^\varepsilon), \\
u^\varepsilon(0)=u_0^\varepsilon \in X_\varepsilon^{\frac{1}{2}},\quad \varepsilon\in [0,\varepsilon_0].
\end{cases}
\end{equation}  

We assume $f$ is continuously differentiable and the equilibrium set of \eqref{semilinear_problem1} for $\varepsilon=0$ is composed of a finite number of hyperbolic equilibrium points. That is 
$$
\mathcal{E}_0:=\{u\in D(A_0):A_0u-f(u)=0\}=\{x_*^{1,0}<x_*^{2,0}\leq ...\leq x_*^{m,0}\}
$$
and $\sigma(A_0-f'(u_*^{i,0}))\cap\{\mu:Re(\mu)=0\}=\emptyset$, for $i\in\{1,...,m\}$.

In order to ensure that all solution of \eqref{semilinear_problem1} are globally defined, and there is a global attractor for the nonlinear semigroup given by theses solutions, we assume the following conditions. 
\begin{itemize}
\item[(i)]  If $n=2$, for all $\eta>0$, there is a constant $C_\eta>0$ such that
$$
 |f(u)-f(v)|\leq C_\eta (e^{\eta|u|^2}+e^{\eta|v|^2})|u-v|,\quad \forall\,u,v\in \R,
$$     
and if $n\geq 3$, there is a constant $\tilde{C}>0$ such that
$$
|f(u)-f(v)|\leq \tilde{C}|u-v| (|u|^{\frac{4}{n-2}}+|v|^{\frac{4}{n-2}}+1 ),\quad \forall\,u,v\in \R.
$$     
\item[(ii)] 
$$
\limsup_{|u|\to\infty} \dfrac{f(u)}{u} <0.
$$
\end{itemize}
Under theses assumptions \cite{J.M.Arrieta1999,J.M.Arrieta2000} and \cite{Hale1988} ensure that the problem \eqref{semilinear_problem1} is globally well posed and generate a nonlinear semigroup satisfying
\begin{equation}\label{nonlinear_semigroup}
T_\varepsilon(t)u_0^\varepsilon=e^{-At}u_0^\varepsilon+\int_0^t e^{-A(t-s)}f(T_\varepsilon(s)u_0^\varepsilon)\,ds,\quad t\geq 0.
\end{equation}
Moreover there is a global attractor $\mathcal{A}_\varepsilon$ for $T_\varepsilon(\cdot)$ uniformly bounded in $X_\varepsilon^\frac{1}{2}$, that is
$$
\sup_{\varepsilon\in [0,\varepsilon_0]}\sup_{w\in \mathcal{A}_\varepsilon}\|w\|_{X_\varepsilon^\frac{1}{2}}<\infty.
$$ 
We also have $T_0(\cdot)$ is a Morse-Smale semigroup and $\mathcal{A}_0=[x_*^{1,0},x_*^{m,0}]$.  

In order to find a rate of convergence for the resolvent operators we consider the projection 
\begin{equation}\label{projection_average}
Pu=\frac{1}{|\Omega|}\int_\Omega u \,dx,\quad u\in L^2(\Omega)\quad\tn{or}\quad u\in X_\varepsilon^\frac{1}{2}.
\end{equation}
Thus $P$ is an orthogonal projection acting on $L^2$ onto $L^2_\Omega$ or $X_\varepsilon^\frac{1}{2}$ onto $X_0^\frac{1}{2}$.
 
\begin{lemma}\label{Rate_of_convergence1}
For $g\in L^2(\Omega)$ with $\|g\|_{L^2}\leq 1$ and $\varepsilon\in (0,\varepsilon_0]$, let $u^\varepsilon$ be the solution of elliptic problem  
$$
\begin{cases}
-\tn{div}(p_\varepsilon(x)\nabla u^\varepsilon)+(\lambda+V_\varepsilon(x))u^\varepsilon=g, \quad x\in \Omega,\\
\dfrac{\partial u^\varepsilon}{\partial \vec{n}}=0,\quad x\in \partial\Omega.
\end{cases}
$$
Then there is a constant $C>0$, independent of $\varepsilon$, such that
$$
\|u^\varepsilon-u^0\|_{X_\varepsilon^\frac{1}{2}}\leq C(\tau(\varepsilon)+p(\varepsilon)^{-\frac{1}{2}}),
$$
where $u^0=\frac{Pg}{\lambda+V_0}$.
\end{lemma} 
\begin{proof}
The weak solution $u^\varepsilon$ satisfies
\begin{equation}\label{weak_large_1}
\int_\Omega p_\varepsilon \nabla u^\varepsilon \nabla \varphi\,dx+\int_\Omega (\lambda+V_\varepsilon)u^\varepsilon \varphi\,dx=\int_\Omega g\varphi\,dx,\quad\forall\,\varphi\in X_\varepsilon^{\frac{1}{2}},\,\,\varepsilon\in(0,\varepsilon_0],
\end{equation}
and
\begin{equation}\label{weak_large_2}
\int_\Omega (\lambda+ V_0)u^0 \varphi\,dx=\int_\Omega Pg \varphi\,dx,\quad\forall\,\varphi\in X_0^{\frac{1}{2}}.
\end{equation}
Taking $\varphi=u^\varepsilon-u^0$ in \eqref{weak_large_1} and $\varphi=Pu^\varepsilon-u^0$ in \eqref{weak_large_2}, we have
$$
\int_\Omega p_\varepsilon |\nabla u^\varepsilon|^2\,dx+\int_\Omega (\lambda+V_\varepsilon)u^\varepsilon(u^\varepsilon-u^0)\,dx=\int_\Omega g(u^\varepsilon-u^0)\,dx;
$$
$$
\int_\Omega (\lambda+V_0)u^0(Pu^\varepsilon-u^0)\,dx=\int_\Omega Pg(Pu^\varepsilon-u^0)\,dx,
$$
which implies
$$
\int_\Omega g(u^\varepsilon-u^0)\,dx-\int_\Omega Pg(Pu^\varepsilon-u^0)\,dx=\int_\Omega g(I-P)u^\varepsilon\,dx
$$
and
\begin{align*}
\int_\Omega p_\varepsilon |\nabla u^\varepsilon|^2\,dx &+\int_\Omega (\lambda+V_\varepsilon )u^\varepsilon(u^\varepsilon-u^0)\,dx-\int_\Omega (\lambda+V_0) u^0(Pu^\varepsilon-u^0)\,dx\\ 
&=\|u^\varepsilon-u^0\|_{X_\varepsilon^\frac{1}{2}}^2+\int_\Omega (V_\varepsilon-V_0)u^0(u^\varepsilon-u^0)\,dx.
\end{align*}
Therefore 
$$
\|u^\varepsilon-u^0\|_{X_\varepsilon^\frac{1}{2}}^2\leq \int_\Omega |V_\varepsilon-V_0||u^0||u^\varepsilon-u^0|\,dx+\int_\Omega |g(I-P)u^\varepsilon|\,dx.
$$
If $n=1$, we have $X_\varepsilon^\frac{1}{2}\subset H^1\subset L^\infty$, thus
$$
\int_\Omega |V_\varepsilon-V_0||u^0||u^\varepsilon-u^0|\,dx\leq C\|u^\varepsilon-u^0\|_{L^\infty}\|V^\varepsilon-V_0\|_{L^1}\leq C\|u^\varepsilon-u^0\|_{X_\varepsilon^\frac{1}{2}}\tau(\varepsilon).
$$
If $n\geq 2$, we have $L^p\subset L^2$, thus
$$
\int_\Omega |V_\varepsilon-V_0||u^0||u^\varepsilon-u^0|\,dx\leq C\|u^\varepsilon-u^0\|_{L^2}\|V^\varepsilon-V_0\|_{L^2}\leq C\|u^\varepsilon-u^0\|_{X_\varepsilon^\frac{1}{2}}\tau(\varepsilon).
$$
By Poincar\'{e}'s inequality for average, we have
$$
\int_\Omega |g(I-P)u^\varepsilon|\,dx\leq \|g\|_{L^2}\Big(\int_\Omega |\nabla u^\varepsilon|^2\,dx\Big)^\frac{1}{2},
$$
but
$$
p(\varepsilon)\int_\Omega|\nabla u^\varepsilon|^2\,dx \leq\int_\Omega p_\varepsilon|\nabla u^\varepsilon-\nabla u^0|^2\,dx\leq \|u^\varepsilon-u^0\|_{X_\varepsilon^\frac{1}{2}}^2.
$$
Put this estimates together the result follows. 
\cqd
\end{proof}

As a consequence of the Lemma \ref{Rate_of_convergence1} the constant of immersion of $H^1(\Omega)\subset X_\varepsilon^\frac{1}{2}$ is not uniform in $\varepsilon$.

\begin{corollary}\label{non_equivalence_norm1}
There is no positive constant $C$ independent of $\varepsilon$ such that 
$$
\|u\|_{X_\varepsilon^\frac{1}{2}}^2\leq C\|u\|_{H^1}^2\quad \forall\,u\in X_\varepsilon^\frac{1}{2}.
$$
\end{corollary}
\begin{proof}
If there is such a constant $C$ take $v^\varepsilon=(I-P)u^\varepsilon$ as given by the Lemma \ref{Rate_of_convergence1}, thus by Poincaré's inequality for average, we have
\begin{align*}
p(\varepsilon)\|v^\varepsilon\|_{H^1}^2 & \leq p(\varepsilon)\int_\Omega |\nabla v^\varepsilon|^2\,dx+\int_\Omega |v^\varepsilon|^2\,dx\\
& \leq Cp(\varepsilon)\int_\Omega |\nabla v^\varepsilon|^2\,dx\\
&\leq C\|v^\varepsilon\|_{X_\varepsilon^\frac{1}{2}}^2\leq C\|v^\varepsilon\|_{H^1}^2.
\end{align*}
\cqd
\end{proof}

The convergence of the resolvent operators can be stated as follows.  
\begin{corollary}
There is a positive constant $C$ independent of $\varepsilon$ such that
$$
\|A_\varepsilon^{-1}-A_0^{-1}P\|_{\LL(L^2,X_\varepsilon^\frac{1}{2})}\leq C(\tau(\varepsilon)+p(\varepsilon)^{-\frac{1}{2}}).
$$ 
\end{corollary}
\begin{proof}
Let $g\in L^2(\Omega)$ such that $\|g\|_{L^2}\leq 1$ and define $u^\varepsilon =A_\varepsilon^{-1}g$ for $\varepsilon\in (0,\varepsilon_0]$ and $u^0=Pg/(\lambda+V_0)$. Now the result follows from Lemma \ref{Rate_of_convergence1}. 
\end{proof}

We are now in position to state the main result in this section.
 
\begin{theorem}\label{estimate_large_diffusion_shadowing}
There is an one dimensional invariant manifold $\mathcal{M}_\varepsilon$ for \eqref{semilinear_problem1} such that $\mathcal{A}_\varepsilon\subset \mathcal{M}_\varepsilon$ and the flow on $\mathcal{A}_\varepsilon$ can be reduced to an ordinary differential equation. Moreover the convergence of attractors of \eqref{nonlinear_semigroup} can be estimate by   
$$
\tn{d}_\varepsilon(\mathcal{A}_\varepsilon,\mathcal{A}_0)\leq C(\tau(\varepsilon)+p(\varepsilon)^{-\frac{1}{2}}).
$$
\end{theorem}
\begin{proof}
If we define $X_0=\R=X_0^\frac{1}{2}$, $X_\varepsilon^\alpha=X_\varepsilon^\frac{1}{2}$ and $X_\varepsilon=L^2(\Omega)$, $\varepsilon\in (0,\varepsilon_0]$, we have $\{A_\varepsilon\}_{\varepsilon\in (0,\varepsilon_0]}$ a family of operators satisfying the assumptions of the Sections \ref{Functional Setting} and \ref{Compact Convergence}, where $A_0$ has spectrum set given by $\sigma(A_0)=\{\lambda+V_0\}$, $\lambda+V_0>0$ and $A_\varepsilon$ has spectrum set given by $\sigma(A_\varepsilon)=\{\lambda_1^\varepsilon,\lambda_2^\varepsilon, \dots\}$, $\lambda_1^\varepsilon\geq m_0>0$. Note that we have $\tn{dim}(X_0)=1$ which implies the limiting ODE \eqref{limit_problem1} one dimensional. Finally we define $E_\varepsilon$ by the inclusion $X_0^\frac{1}{2}\subset X_\varepsilon^\frac{1}{2}$ and $M_\varepsilon=P$, where $P$ is defined in \eqref{projection_average}. Hence the assumptions of the Section \ref{Rate of Convergence} and \ref{Shadowing Theory and Rate of Convergence} are satisfied. The proof is completed.
\cqd
\end{proof}

Due to the simplicity of the dynamics of the Problem \ref{perturbed_problem1} we can prove the Theorem \ref{estimate_large_diffusion_shadowing} without using the Theorem \ref{Shadowing_reaction_diffusion}, that is, without Shadowing Theory. In fact, we explore the geometry offer by the finite dimension. This geometric argument that motivated us to investigate the rate of convergence of attractors for problems which the asymptotic behavior can be described by a system of ordinary differential equation.

First we estimate the convergence of equilibrium points.
\begin{theorem}
Let $u_*^0\in\mathcal{E}_0$. Then for $\varepsilon$ sufficiently small (we still denote $\varepsilon\in (0,\varepsilon_0]$), there is $\delta>0$ such that the equation $A_\varepsilon u-f(u)=0$ has the only solution $u_*^\varepsilon\in \{u\in X_\varepsilon^\frac{1}{2}\,;\,\|u-u_*^0\|_{X_\varepsilon^\frac{1}{2}}\leq \delta\}$. Moreover
\begin{equation}\label{rate_of_equilibrium}
\|u_*^\varepsilon-u_*^0\|_{X_\varepsilon^\frac{1}{2}}\leq C(\tau(\varepsilon)+p(\varepsilon)^{-\frac{1}{2}}).
\end{equation} 
\end{theorem}
\begin{proof}
The proof is the same as given in \cite{Arrieta} and \cite{Carbone2008}. Here we just need to proof the estimates \eqref{rate_of_equilibrium}. We have $u_*^\varepsilon$ and $u_*^0$ given by 
$$
u_*^0=(A_0+V_0)^{-1}[f(u_*^0)+V_0u_*^0]\quad\tn{and}\quad u_*^\varepsilon=(A_\varepsilon+V_0)^{-1}[f(u_*^\varepsilon)+V_0 u_*^\varepsilon],
$$ 
where $V_0=-f'(u_*^0)$. Thus
\begin{align*}
\|u_*^\varepsilon-u_*^0\|_{X_\varepsilon^\frac{1}{2}}&\leq \|(A_\varepsilon+V_0)^{-1}[f(u_*^\varepsilon)+V_0 u_*^\varepsilon]-(A_0+V_0)^{-1}[f(u_*^0)+V_0u_*^0]\|_{X_\varepsilon^\frac{1}{2}}\\
&\leq \|[(A_\varepsilon+V_0)^{-1}-(A_0+V_0)^{-1}P][f(u_*^\varepsilon)+V_0 u_*^\varepsilon]\|_{X_\varepsilon^\frac{1}{2}}\\
&+\|(A_0+V_0)^{-1}P[f(u_*^\varepsilon)-f(u_*^0)+V_0(u_*^\varepsilon-u_*^0)]\|_{X_\varepsilon^\frac{1}{2}}.
\end{align*}
We have the following equality 
$$
(A_\varepsilon+V_0)^{-1}-(A_0+V_0)^{-1}P=[I-(A_\varepsilon+V_0)^{-1}V_0](A_\varepsilon^{-1}-A_0^{-1}P)[I-V_0(A_0+V_0)^{-1}].
$$
And then
$
\|[(A_\varepsilon+V_0)^{-1}-(A_0+V_0)^{-1}P][f(u_*^\varepsilon)+V_0 u_*^\varepsilon]\|_{X_\varepsilon^\frac{1}{2}}\leq C(\tau(\varepsilon)+p(\varepsilon)^{-\frac{1}{2}}).
$

If we denote $z^\varepsilon=f(u_*^\varepsilon)-f(u_*^0)+V_0(u_*^\varepsilon-u_*^0)$, since $f$ is continuously differentiable, for all $\delta>0$ there is $\varepsilon$ sufficiently small such that 
$
\|z^\varepsilon\|_{X_\varepsilon^\frac{1}{2}}\leq \delta\|u_*^\varepsilon-u_*^0\|_{X_\varepsilon^\frac{1}{2}}.
$
Thus
$$
\|(A_0+V_0)^{-1}Pz^\varepsilon\|_{X_\varepsilon^\frac{1}{2}}\leq \delta \|(A_0+V_0)^{-1}P\|_{\LL(L^2,X_\varepsilon^\frac{1}{2})}\|u_*^\varepsilon-u_*^0\|_{X_\varepsilon^\frac{1}{2}}. 
$$
We choice $\delta$ sufficiently small such that $\delta\|(A_0+V_0)^{-1}P\|_{\LL(L^2,X_\varepsilon^\frac{1}{2})}\leq\frac{1}{2}$, and then
$$
\|u_*^\varepsilon-u_*^0\|_{X_\varepsilon^\frac{1}{2}}\leq C(\tau(\varepsilon)+p(\varepsilon)^{-\frac{1}{2}})+\frac{1}{2}\|u_*^\varepsilon-u_*^0\|_{X_\varepsilon^\frac{1}{2}}.
$$
\cqd
\end{proof}

We now give a geometric prove of the Theorem \ref{estimate_large_diffusion_shadowing}.

\begin{figure}[h]
\begin{center}
\includegraphics[scale=0.60]{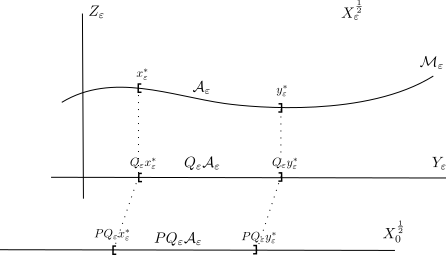} 
\caption{Geometry of the phase space}
\end{center}
\end{figure}

\begin{proof} By triangle inequality, we have
\begin{align*}
\tn{d}_\varepsilon(\mathcal{A}_\varepsilon,\mathcal{A}_0)&\leq \tn{d}_\varepsilon(\mathcal{A}_\varepsilon,Q_\varepsilon\mathcal{A}_\varepsilon)+
\tn{d}_\varepsilon(Q_\varepsilon\mathcal{A}_\varepsilon,\mathcal{A}_0)\\
& \leq \tn{d}_\varepsilon(\mathcal{A}_\varepsilon,Q_\varepsilon\mathcal{A}_\varepsilon)+
\tn{d}_\varepsilon(Q_\varepsilon\mathcal{A}_\varepsilon,PQ_\varepsilon\mathcal{A}_\varepsilon)+
\tn{d}_\varepsilon(PQ_\varepsilon\mathcal{A}_\varepsilon,\mathcal{A}_0).
\end{align*}

We estimate each part. 

Let $z\in\mathcal{A}_\varepsilon$, since $\mathcal{A}_\varepsilon\subset\mathcal{M}_\varepsilon$, $z=z^\varepsilon+s_*^\varepsilon(z^\varepsilon)$ for some $z^\varepsilon\in Q_\varepsilon\mathcal{A}_\varepsilon$, then
\begin{align*}
\tn{dist}(z,Q_\varepsilon\mathcal{A}_\varepsilon)&=\inf_{x\in Q_\varepsilon(\bar{\lambda})\mathcal{A}_\varepsilon}\|z-x\|_{X_\varepsilon^\frac{1}{2}}\leq \|z^\varepsilon+s_*^\varepsilon(z^\varepsilon)-z^\varepsilon\|_{X_\varepsilon^\frac{1}{2}}=\|s_*^\varepsilon(z^\varepsilon)\|_{X_\varepsilon^\frac{1}{2}}\\
&\leq |\!|\!|s_*^\varepsilon|\!|\!|\leq C(\tau(\varepsilon)+p(\varepsilon)^{-\frac{1}{2}}),
\end{align*} 
\begin{align*}
\tn{dist}(z^\varepsilon,\mathcal{A}_\varepsilon)&=\inf_{x\in \mathcal{A}_\varepsilon}\|z^\varepsilon
-x\|_{X_\varepsilon^\frac{1}{2}}\leq \|z^\varepsilon-(z^\varepsilon+s_*^\varepsilon(z^\varepsilon))\|_{X_\varepsilon^\frac{1}{2}}=\|s_*^\varepsilon(z^\varepsilon)\|_{X_\varepsilon^\frac{1}{2}}\\
&\leq |\!|\!|s_*^\varepsilon|\!|\!|\leq C(\tau(\varepsilon)+p(\varepsilon)^{-\frac{1}{2}}),
\end{align*} 
which implies 
$$
\tn{d}_\varepsilon(\mathcal{A}_\varepsilon,Q_\varepsilon\mathcal{A}_\varepsilon)= \max\{\tn{dist}_\varepsilon(\mathcal{A}_\varepsilon,Q_\varepsilon\mathcal{A}_\varepsilon),
\tn{dist}_\varepsilon(Q_\varepsilon\mathcal{A}_\varepsilon,\mathcal{A}_\varepsilon)\}\leq C(\tau(\varepsilon)+p(\varepsilon)^{-\frac{1}{2}}).
$$
 
Let $z^\varepsilon\in Q_\varepsilon\mathcal{A}_\varepsilon$ then $z^\varepsilon=Q_\varepsilon w^\varepsilon$ for some $w^\varepsilon\in\mathcal{A}_\varepsilon$. Since $\bigcup_{\varepsilon\in[0,\varepsilon_0]}\mathcal{A}_\varepsilon$ is uniformly bounded in $X_\varepsilon^\frac{1}{2}$ we can assume $\|w^\varepsilon\|_{X_\varepsilon^\frac{1}{2}}\leq C$. Thus
\begin{align*}
\tn{dist}(z^\varepsilon,PQ_\varepsilon\mathcal{A}_\varepsilon)&=\inf_{x\in PQ_\varepsilon\mathcal{A}_\varepsilon}\|z^\varepsilon
-x\|_{X_\varepsilon^\frac{1}{2}}
\leq \|z^\varepsilon-Pz^\varepsilon\|_{X_\varepsilon^\frac{1}{2}}=\|Q_\varepsilon w^\varepsilon-PQ_\varepsilon w^\varepsilon\|_{X_\varepsilon^\frac{1}{2}}\\
&=\|(Q_\varepsilon-P)Q_\varepsilon w^\varepsilon\|_{X_\varepsilon^\frac{1}{2}}\leq C(\tau(\varepsilon)+p(\varepsilon)^{-\frac{1}{2}}),
\end{align*} 
\begin{align*}
\tn{dist}(Pz^\varepsilon,Q_\varepsilon\mathcal{A}_\varepsilon)&=\inf_{x\in Q_\varepsilon\mathcal{A}_\varepsilon}\|Pz^\varepsilon
-x\|_{X_\varepsilon^\frac{1}{2}}
\leq \|Pz^\varepsilon-z^\varepsilon\|_{X_\varepsilon^\frac{1}{2}}=\|PQ_\varepsilon w^\varepsilon-Q_\varepsilon w^\varepsilon\|_{X_\varepsilon^\frac{1}{2}}\\
&=\|(P-Q_\varepsilon) Q_\varepsilon w^\varepsilon\|_{X_\varepsilon^\frac{1}{2}}\leq C(\tau(\varepsilon)+p(\varepsilon)^{-\frac{1}{2}}),
\end{align*} 
which implies 
\begin{align*}
\tn{d}_\varepsilon(Q_\varepsilon\mathcal{A}_\varepsilon,PQ_\varepsilon\mathcal{A}_\varepsilon)&=\max\{\tn{dist}_\varepsilon(Q_\varepsilon\mathcal{A}_\varepsilon,PQ_\varepsilon\mathcal{A}_\varepsilon),
\tn{dist}_\varepsilon(PQ_\varepsilon\mathcal{A}_\varepsilon,Q_\varepsilon\mathcal{A}_\varepsilon)\}\\
&\leq C(\tau(\varepsilon)+p(\varepsilon)^{-\frac{1}{2}}).
\end{align*}

Finally we have $\mathcal{A}_0=[x_*^{1,0},x_*^{m,0}]$ and $T_0(\cdot)$ is a Morse-Smale semigroup. Since Morse-Smale semigroup are stable (see \cite{Bortolan}) we can assume $\mathcal{A}_\varepsilon=[x_*^{1,\varepsilon},x_*^{m,\varepsilon}]$ which implies $Q_\varepsilon\mathcal{A}_\varepsilon=[Q_\varepsilon x_*^{1,\varepsilon},Q_\varepsilon x_*^{m,\varepsilon}]$ and $PQ_\varepsilon \mathcal{A}_\varepsilon=[PQ_\varepsilon x_*^{1,\varepsilon},PQ_\varepsilon x_*^{m,\varepsilon}]$. Without loss of generality we can assume $PQ_\varepsilon x_*^{1,\varepsilon}\leq x_*^{1,0}$ and $PQ_\varepsilon x_*^{m,\varepsilon}\leq x_*^{m,0}$. Thus
$$
\tn{d}_\varepsilon(PQ_\varepsilon\mathcal{A}_\varepsilon,\mathcal{A}_0)=\|PQ_\varepsilon x_*^{1,\varepsilon}- x_*^{1,0}\|_{X_\varepsilon^\frac{1}{2}}+\|PQ_\varepsilon x_*^{m,\varepsilon}- x_*^{m,0}\|_{X_\varepsilon^\frac{1}{2}}.
$$
But
\begin{align*}
\|PQ_\varepsilon x_*^{1,\varepsilon}- x_*^{1,0}\|_{X_\varepsilon^\frac{1}{2}}&=\|PQ_\varepsilon x_*^{1,\varepsilon}- Px_*^{1,0}\|_{X_\varepsilon^\frac{1}{2}}\leq \|P\|_{\LL(L^2)}\|Q_\varepsilon x_*^{1,\varepsilon}- Px_*^{1,0}\|_{X_\varepsilon^\frac{1}{2}}\\
&\leq \|P\|_{\LL(L^2)}\|Q_\varepsilon x_*^{1,\varepsilon}- Px_*^{1,\varepsilon}\|_{X_\varepsilon^\frac{1}{2}}+\|P\|_{\LL(L^2)}\|Px_*^{1,\varepsilon}-Px_*^{1,0}\|_{X_\varepsilon^\frac{1}{2}}\\
&\leq C(\tau(\varepsilon)+p(\varepsilon)^{-\frac{1}{2}}).
\end{align*}

In the same way
$
\|PQ_\varepsilon x_*^{m,\varepsilon}- x_*^{m,0}\|_{X_\varepsilon^\frac{1}{2}}\leq C(\tau(\varepsilon)+p(\varepsilon)^{-\frac{1}{2}}).
$
\cqd
\end{proof}

\section{large diffusion except in a neighborhood of a point}\label{A cell tissue reaction-diffusion problem}
In this section we consider a class of diffusion coefficients $p_\varepsilon\in C^1([0,1])$ that are large except in a neighborhood of a point where they become small. This kind of problem has been studied in the works \cite{Carvalho1994} and \cite{Fusco1987} where it has proved that the dynamic is dictated by a ordinary differential equation. We obtain a rate of convergence that enable us knowing how fast the dynamic approaches one ordinary differential equation when the parameter $\varepsilon$ varies.  

Consider the scalar parabolic problem 
\begin{equation}\label{scalar_problem}
\begin{cases}u^\varepsilon_t-(p_\varepsilon(x) u^\varepsilon_x)_x+\lambda u^\varepsilon=f(u^\varepsilon),\quad 0<x<1,\,t>0,\\u^\varepsilon(0)=0=u^\varepsilon(1),\quad t>0,       \end{cases}
\end{equation}
where $\varepsilon\in (0,\varepsilon_0]$, $0<\varepsilon_0\leq 1$, $\lambda>0$ and $f:\R\to\R$ is continuously differentiable and satisfies the following dissipativeness condition
\begin{equation}\label{dissipativeness_condition}
\limsup_{|u|\to \infty}\frac{f(u)}{u}\leq -r,
\end{equation}
for some $r>0$. Assume that the diffusion is large except in a neighborhood of a point where it becomes small, that is, $p_\varepsilon\in C^2([0,1])$ and for positive constants $e_1,e_2, l_1$ and $a_1$, is valid
$$
\begin{cases} p_\varepsilon(x)\geq\frac{e_1}{\varepsilon},\quad x\in [0,x_1-\varepsilon l_1'],\\ p_\varepsilon(x)\geq \frac{e_2}{\varepsilon},\quad x\in[x_1+\varepsilon l_1',1],\\p_\varepsilon(x)\geq \varepsilon a_1,\quad x\in [x_1-\varepsilon l_1',x_1+\varepsilon l_1'],\\ p_\varepsilon(x)\leq \varepsilon a_1'\quad x\in [x_1-\varepsilon l_1,x_1+\varepsilon l_1],  \end{cases}
$$
where $0<x_1<1$ and $l_1',a_1'$ are functions of $\varepsilon$ that approach $l_1,a_1$ respectively from above as $\varepsilon\to 0$ (see Fig \eqref{diffusion1}).

\begin{figure}[h]
\begin{center}
\includegraphics[scale=0.7]{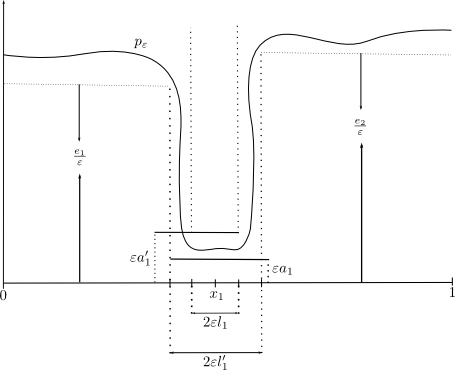} 
\caption{Diffusion}\label{diffusion1}
\end{center}
\end{figure}

We follow \cite{Carvalho1994} and \cite{Fusco1987} in order to determine the asymptotic behavior of \eqref{scalar_problem}. It is expect that the solutions converge to a constant (possibly different) on each of the interval $(0,x_1)$ and $(x_1,1)$, which implies that the equation \eqref{scalar_problem} has the same dynamic as an ordinary differential equation in $\R^2$. 

We start by defining the operator $A_\varepsilon:D(A_\varepsilon)\subset L^2(0,1)\to L^2(0,1)$ by
$$
\mathcal{D}(A_\varepsilon)=\{u\in H^2(0,1)\,;\, u_x(0)=u_x(1)=0\}\quad\tn{and}\quad A_\varepsilon u=-(p_\varepsilon u_x)_x+\lambda u.
$$
We have $A_\varepsilon$ is a positive, self-adjoint and invertible operator with compact resolvent for each $\varepsilon\in (0,\varepsilon_0]$. Thus, we can define, in the usual way (see\cite{Henry1980}), the fractional power space $X_\varepsilon^\frac{1}{2}=H^1(0,1)$, $\varepsilon\in (0,\varepsilon_0]$.

The proof of the next Lemma follows the same arguments in the Section 2 in \cite{Carvalho1994}.

\begin{lemma}
Let $\sigma(A_\varepsilon)=\{\lambda,\lambda_2^\varepsilon,\lambda_3^\varepsilon,...\}$ be the sequence of eigenvalues of operator $A_\varepsilon$ and let $\{1,\varphi_2^\varepsilon,\varphi_3^\varepsilon,...\}$ be the corresponding sequence of normalized eigenfunctions. Then
$$
\lambda_2^\varepsilon\overset{\varepsilon\to 0}\longrightarrow \frac{a_1}{2l_1x_1(1-x_1)}+\lambda,\quad \lambda_j^\varepsilon\overset{\varepsilon\to 0}\longrightarrow\infty,\,\,\,j\geq 3,
$$
and
$$
\varphi_2^\varepsilon\overset{\varepsilon\to 0}\longrightarrow \begin{cases} -\sqrt{\frac{1-x_1}{x_1}}\,\,x\in (0,x_1),\\ \sqrt{\frac{x_1}{1-x_1}},\,\,x\in (x_1,1).  \end{cases}
$$
We can decompose the fractional power space $X_\varepsilon^\frac{1}{2}=\tn{span}[1,\varphi_2^\varepsilon]\oplus(\tn{span}[1,\varphi_2^\varepsilon])^\perp$ which implies that the limiting problem of \eqref{scalar_problem} is given by following ordinary differential equation in $\R^2$
\begin{equation}\label{limiting_problem_matrix}
\begin{bmatrix}
\dot{u}_1^0 \\
\dot{u}_2^0
\end{bmatrix}+
\begin{bmatrix}
\frac{a_1}{2l_1 x_1}+\lambda & -\frac{a_1}{2l_1x_1} \\
-\frac{a_1}{2l_1(1-x_1)} & \frac{a_1}{2l_1 (1-x_1)}+\lambda
\end{bmatrix}
\begin{bmatrix}
u_1^0 \\
u_2^0
\end{bmatrix}=
\begin{bmatrix}
f(u_1^0) \\
f(u_2^0)
\end{bmatrix}.
\end{equation}
\end{lemma}

In what follows we denote $X_0^\frac{1}{2}$ as $\R^2$ with the norm
$$
\|u^0\|_{X_0^\frac{1}{2}}^2=\langle u_1^0, u_2^0 \rangle_{X_0^\frac{1}{2}}=|u_1^0|^2x_1+|u^0_2|^2(1-x_1)
$$  
and we define the operator $A_0:X_0^\frac{1}{2}\to X_0^\frac{1}{2}$ given by
$$
A_0u^0=\begin{bmatrix}
\frac{a_1}{2l_1 x_1}+\lambda & -\frac{a_1}{2l_1x_1} \\
-\frac{a_1}{2l_1(1-x_1)} & \frac{a_1}{2l_1 (1-x_1)}+\lambda
\end{bmatrix}
\begin{bmatrix}
u_1^0 \\
u_2^0
\end{bmatrix}.
$$

If we denote the Nemitskii functional of $f$ by the same notation $f$, then \eqref{scalar_problem} and \eqref{limiting_problem_matrix} can be written as  
\begin{equation}\label{coupled_semilinear_problem}
\begin{cases}
u^\varepsilon_t+A_\varepsilon u^\varepsilon = f(u^\varepsilon), \\
u^\varepsilon(0)=u_0^\varepsilon \in X_\varepsilon^{\frac{1}{2}},\quad \varepsilon\in [0,\varepsilon_0].
\end{cases}
\end{equation}  
Since we have the condition \eqref{dissipativeness_condition} it follows from \cite{J.M.Arrieta1999,J.M.Arrieta2000,J.M.Arrieta2000a} and standard arguments of ordinary differential equations that we can consider $f$ bounded globally Lipschitz such that the problem \eqref{coupled_semilinear_problem}, for each $\varepsilon\in [0,\varepsilon_0]$, is well posed for positive time and the solutions are continuously differentiable with respect to the initial data.  Hence we are able to consider in $X_\varepsilon^{\frac{1}{2}}$ the family of nonlinear semigroups $\{T_\varepsilon(\cdot)\}_{\varepsilon\in [0,\varepsilon_0]}$ defined by $T_\varepsilon(t)=u^\varepsilon(t,u^\varepsilon_0)$, $t\geq 0$, where $u^\varepsilon(t,u^\varepsilon_0)$ is the solution of \eqref{coupled_semilinear_problem} through $u^\varepsilon_0\in X_\varepsilon^\frac{1}{2}$ and 

$$
T_\varepsilon(t)u_0^\varepsilon   = e^{-A_{\varepsilon}t}u_0^\varepsilon+\int_0^{t} e^{-A_\varepsilon (t-s)}f(T_\varepsilon(s))\,ds,\quad t\geq 0, 
$$
has a global attractor $\mathcal{A}_\varepsilon$, for each $\varepsilon\in (0,\varepsilon_0]$ such that $\overline{\bigcup_{\varepsilon\in (0,\varepsilon_0]}\mathcal{A}_\varepsilon}$ is compact and uniformly bounded. 

We assume that \eqref{coupled_semilinear_problem} with $\varepsilon=0$ has a global attractor $\mathcal{A}_0\subset X_0^\frac{1}{2}$ and $T_0(\cdot)$ is a Morse-Smale semigroup in $X_0^\frac{1}{2}.$ 

In order to understand the problems in the same space we define the operators 
$E_\varepsilon:\R^2\to X_\varepsilon^\frac{1}{2}$ and $M_\varepsilon:L^2\to\R^2$ by
$$
E_\varepsilon (u_1,u_2)=\begin{cases}u_1,\quad x\in [0,x_1-\varepsilon l_1],\\u_1+\frac{u_2-u_1}{\varepsilon 2 l_1}(x-x_1+\varepsilon l_1),\,\,x\in[x_1-\varepsilon l_1,x_1+\varepsilon l_1],   \\u_2,\quad x\in [x_1+\varepsilon l_1,1],\end{cases}
$$
$$
M_\varepsilon u=(u_1,u_2),\,\, u_1=\frac{1}{x_1-\varepsilon l_1}\int_0^{x_1-\varepsilon l_1}u\,dx,\,\, u_2=\frac{1}{1-(x_1+\varepsilon l_1)}\int_{x_1+\varepsilon l_1}^1 u\,dx. 
$$
The next lemma is essential to obtain the rate of convergence of the resolvent operators.

\begin{lemma}
For $g\in L^2$ with $\|g\|_{L^2}\leq 1$ and $\varepsilon\in (0,\varepsilon_0]$, let $u^\varepsilon$ be the solution of elliptic problem  $A_\varepsilon u^\varepsilon=g$
and let $u^0=(u_1^0,u_2^0)$ be the solution of $A_0 u^0=M_\varepsilon g$. Then there is a constant $C>0$, independent of $\varepsilon$, such that
\begin{equation}\label{convergence_resolvent0}
\|u^\varepsilon-E_\varepsilon u^0\|^2_{X_\varepsilon^\frac{1}{2}}\leq C(|a_1'-a_1|+|l_1'-l_1|+ \varepsilon^\frac{1}{2}).
\end{equation} 
\end{lemma}

\begin{proof}
The arguments used here was inspired in \cite{Carvalho2010a}.

We denote $E=E_\varepsilon$.  The weak solution $u^\varepsilon$ satisfies,
\begin{equation}\label{weak_solution1}
\int_0^1 p_\varepsilon u^\varepsilon_x \varphi_x \,dx+\int_0^1\lambda u^\varepsilon \varphi\,dx=\int_0^1 g \varphi \,dx,\quad\forall \varphi\in X_\varepsilon^\frac{1}{2},\quad \varepsilon\in (0,\varepsilon_0],
\end{equation}
and
$$
\frac{a_1}{2l_1} u^0_1 \varphi_1^0+\lambda u_1^0\varphi^0_1 x_1-\frac{a_1}{2l_1} u^0_2 \varphi_1^0-\frac{a_1}{2l_1} u^0_1 \varphi_2^0+\frac{a_1}{2l_1} u^0_2 \varphi_1^2+\lambda u_2^0\varphi^0_2(1-x_1)=g_1^0u_1^0x_1+g_2^0u_2^0(1-x_1),
$$
for all $\varphi^0=(\varphi_1^0,\varphi_2^0)\in \R^2$.

We denote
$$
\mu_\varepsilon=\min_{u\in X_\varepsilon^\frac{1}{2}}\Big\{\frac{1}{2}\int_0^1 p_\varepsilon |u_x|^2 \,dx+\frac{1}{2}\int_0^1\lambda |u|^2\,dx-\int_0^1 g u\,dx\Big\},\quad \varepsilon\in (0,\varepsilon_0];
$$
$$
\mu_0=\min_{u\in \R^2}\Big\{\frac{1}{2}\frac{a_1}{2l_1}(u_2-u_1)^2+\frac{1}{2}\lambda |u_1|^2x_1+\frac{1}{2}\lambda |u_2|^2(1-x_1)-g_1u_1x_1-g_2u_2(1-x_1)\Big\}.    
$$
Then 

\begin{align*}
\mu_\varepsilon &=\frac{1}{2}\int_0^1 p_\varepsilon |u_x^\varepsilon|^2 \,dx+\frac{1}{2}\int_0^1\lambda |u^\varepsilon|^2\,dx-\int_0^1 g u^\varepsilon\,dx\\
&=\frac{1}{2}\int_0^1 p_\varepsilon |u_x^\varepsilon-(Eu^0)_x+(Eu^0)_x|^2 \,dx+\frac{1}{2}\int_0^1\lambda |u^\varepsilon-Eu^0+Eu^0|^2\,dx\\
&-\int_0^1 g (u^\varepsilon-Eu^0+Eu^0)\,dx\\
&=\frac{1}{2}\int_0^1 p_\varepsilon |u_x^\varepsilon-(Eu^0)_x|^2\,dx+\int_0^1 p_\varepsilon (u_x^\varepsilon-(Eu^0)_x)((Eu^0)_x)\,dx+\frac{1}{2}\int_0^1 p_\varepsilon|(Eu^0)_x|^2\,dx\\
&+\frac{1}{2}\int_0^1 \lambda |u^\varepsilon-Eu^0|^2\,dx+\int_0^1 \lambda (u^\varepsilon-Eu^0)(Eu^0)\,dx+\frac{1}{2}\int_0^1 \lambda|Eu^0|^2\,dx\\
&-\int_0^1 g (u^\varepsilon-Eu^0)\,dx-\int_0^1 gEu^0\,dx.
\end{align*}
We can write

$$
\int_0^1 p_\varepsilon (u_x^\varepsilon-(Eu^0)_x)((Eu^0)_x)\,dx=-\frac{1}{2}\int_0^1 p_\varepsilon |u_x^\varepsilon-(Eu^0)_x|^2\,dx+\frac{1}{2}\int_0^1 p_\varepsilon (u_x^\varepsilon-(Eu^0)_x)u^\varepsilon_x\,dx;
$$

$$
\int_0^1 \lambda (u^\varepsilon-Eu^0)(Eu^0)\,dx=-\frac{1}{2}\int_0^1 \lambda |u^\varepsilon-Eu^0|^2\,dx+\frac{1}{2}\int_0^1 \lambda (u^\varepsilon-Eu^0)u^\varepsilon\,dx.
$$
Taking $\varphi=u^\varepsilon-Eu^0$ in \eqref{weak_solution1}, we have

$$
\frac{1}{2}\int_0^1 p_\varepsilon (u_x^\varepsilon-(Eu^0)_x)u^\varepsilon_x\,dx+\frac{1}{2}\int_0^1 \lambda (u^\varepsilon-Eu^0)u^\varepsilon\,dx=\int_0^1 g (u^\varepsilon-Eu^0)\,dx.
$$
Thus

$$
\mu_\varepsilon=-\frac{1}{2}\|u^\varepsilon-Eu^0\|_{X_\varepsilon^\frac{1}{2}}^2+\frac{1}{2}\int_0^1 p_\varepsilon|(Eu^0)_x|^2\,dx+\frac{1}{2}\int_0^1 \lambda|Eu^0|^2\,dx-\int_0^1 gEu^0\,dx.
$$
But

\begin{align*}
&I_1:=\frac{1}{2}\int_0^1 p_\varepsilon|(Eu^0)_x|^2\,dx+\frac{1}{2}\int_0^1 \lambda|Eu^0|^2\,dx-\int_0^1 gEu^0\,dx\\
&=\frac{1}{2}\int_{x_1-\varepsilon l_1}^{x_1+\varepsilon l_1} p_\varepsilon\Big[\frac{u^0_2-u^0_1}{\varepsilon 2l_1}\Big]^2\,dx+\frac{1}{2}\int_0^{x_1-\varepsilon l_1} \lambda|u^0_1|^2\,dx+\frac{1}{2}\int_{x_1-\varepsilon l_1}^{x_1+\varepsilon l_1} \lambda|Eu^0|^2\,dx\\
&+\frac{1}{2}\int_{x_1+\varepsilon l_1}^1 \lambda|u^0_2|^2\,dx-\int_0^{x_1-\varepsilon l_1} g_1u^0_1\,dx+\int_0^{x_1-\varepsilon l_1}(g_1-g)u^0_1\,dx-\int_{x_1-\varepsilon l_1}^{x_1+\varepsilon l_1} gEu^0\,dx\\
&-\int_{x_1+\varepsilon l_1}^1 g_2u^0_2\,dx+\int_{x_1+\varepsilon l_1}^1(g_2-g)u^0_2\,dx
\end{align*}
and we can write

$$
\int_0^{x_1-\varepsilon l_1}(g_1-g)u^0_1\,dx=u^0_1\int_0^{x_1}(g_1-g)\,dx-u^0_1\int_{x_1-\varepsilon l_1}^{x_1}(g_1-g)\,dx=-\int_{x_1-\varepsilon l_1}^{x_1}(g_1-g)u_1^0\,dx;
$$

$$
\int_{x_1+\varepsilon l_1}^1(g_2-g)u^0_2\,dx=u^0_2\int_{x_1}^1(g_2-g)\,dx-u^0_2\int_{x_1}^{x_1+\varepsilon l_1}(g_2-g)\,dx=-\int_{x_1}^{x_1+\varepsilon l_1}(g_2-g)u^0_2\,dx.
$$
Thus

\begin{align*}
I_1&\leq \frac{1}{2}\frac{a_1'|u_2^0-u_1^0|^2}{(2l_1)^2\varepsilon}\int_{x_1-\varepsilon l_1}^{x_1+\varepsilon l_1}\,dx+\frac{1}{2}\lambda |u_1^0|^2(x_1-\varepsilon l_1)+\frac{1}{2}\int_{x_1-\varepsilon l_1}^{x_1+\varepsilon l_1} \lambda|Eu^0|^2\,dx\\
&+\frac{1}{2}\lambda |u_2^0|^2(1-x_1-\varepsilon l_1)-g_1^0u_1^0(x_1-\varepsilon l_1)-\int_{x_1-\varepsilon l_1}^{x_1+\varepsilon l_1} gEu^0\,dx-g_2^0u_2^0(1-x_1-\varepsilon l_1)\\
&-\int_{x_1-\varepsilon l_1}^{x_1}(g_1-g)u_1^0\,dx-\int_{x_1}^{x_1+\varepsilon l_1}(g_2-g)u^0_2\,dx\\
&\leq \mu_0+\frac{1}{2}\frac{(a_1-a_1')|u_2^0-u_1^0|^2}{2l_1}+\frac{1}{2}\int_{x_1-\varepsilon l_1}^{x_1+\varepsilon l_1} \lambda|Eu^0|^2\,dx-\frac{1}{2}\lambda|u_1^0|^2\varepsilon l_1-\frac{1}{2}\lambda|u_2^0|^2\varepsilon l_1\\
&+g_1^0u_1^0\varepsilon l_1-\int_{x_1-\varepsilon l_1}^{x_1+\varepsilon l_1} gEu^0\,dx         +g_2^0u_2^0\varepsilon l_1-\int_{x_1-\varepsilon l_1}^{x_1}(g_1-g)u_1^0\,dx-\int_{x_1}^{x_1+\varepsilon l_1}(g_2-g)u^0_2\,dx.
\end{align*}
Therefore

$$
\frac{1}{2}\|u^\varepsilon-E u^0\|_{X_\varepsilon^\frac{1}{2}}^2\leq \mu_0-\mu_\varepsilon +C(|a_1'-a_1|+\varepsilon^\frac{1}{2}).
$$
On the other hand,

\begin{align*}
\mu_\varepsilon &=\frac{1}{2}\int_0^1 p_\varepsilon |u_x^\varepsilon|^2 \,dx+\frac{1}{2}\int_0^1\lambda |u^\varepsilon|^2\,dx-\int_0^1 g u^\varepsilon\,dx\\
&=\frac{1}{2}\int_0^{x_1-\varepsilon l_1'} p_\varepsilon |u_x^\varepsilon|^2 \,dx+\frac{1}{2}\int_{x_1-\varepsilon l_1'}^{x_1+\varepsilon l_1'} p_\varepsilon |u_x^\varepsilon|^2 \,dx+\int_{x_1+\varepsilon l_1'}^1 p_\varepsilon |u_x^\varepsilon|^2 \,dx\\
&+\frac{1}{2}\int_0^{x_1-\varepsilon l_1'}\lambda |u^\varepsilon|^2\,dx+\frac{1}{2}\int_{x_1-\varepsilon l_1'}^{x_1+\varepsilon l_1'}\lambda |u^\varepsilon|^2\,dx+\frac{1}{2}\int_{x_1+\varepsilon l_1'}^1\lambda |u^\varepsilon|^2\,dx\\
&-\int_0^{x_1-\varepsilon l_1'} g u^\varepsilon\,dx-\int_{x_1-\varepsilon l_1'}^{x_1+\varepsilon l_1'} g u^\varepsilon\,dx-\int_{x_1+\varepsilon l_1'}^1 g u^\varepsilon\,dx.
\end{align*}
By H\"{o}lder's inequality, we have

$$
\frac{1}{2}\frac{a_1}{2l_1'}\Big(\int_{x_1-\varepsilon l_1'}^{x_1+\varepsilon l_1'} u_x^\varepsilon \,dx\Big)^2\leq \frac{1}{2}\frac{a_1}{2l_1'}\int_{x_1-\varepsilon l_1'}^{x_1+\varepsilon l_1'} |u_x^\varepsilon|^2 \,dx2\varepsilon l_1'\leq\frac{1}{2}\int_{x_1-\varepsilon l_1'}^{x_1+\varepsilon l_1'} p_\varepsilon |u_x^\varepsilon|^2 \,dx, 
$$
and we also have

$$
\frac{1}{2}\int_0^{x_1-\varepsilon l_1'}\lambda |u^\varepsilon|^2\,dx\geq \int_0^{x_1-\varepsilon l_1'}\lambda(u^\varepsilon-u^\varepsilon_1)u^\varepsilon_1\,dx+\frac{1}{2}\int_0^{x_1-\varepsilon l_1'}\lambda|u^\varepsilon_1|^2\,dx;
$$

$$
\frac{1}{2}\int_{x_1+\varepsilon l_1'}^1\lambda |u^\varepsilon|^2\,dx\geq \int_{x_1+\varepsilon l_1'}^1\lambda(u^\varepsilon-u^\varepsilon_2)u^\varepsilon_2\,dx+\frac{1}{2}\int_{x_1+\varepsilon l_1'}^1\lambda|u^\varepsilon_2|^2\,dx.
$$
Thus
\begin{align*}
\mu_\varepsilon & \geq \frac{1}{2}\frac{a_1}{2l_1'}\Big(\int_{x_1-\varepsilon l_1'}^{x_1+\varepsilon l_1'} u_x^\varepsilon \,dx\Big)^2 +\int_0^{x_1-\varepsilon l_1'}\lambda(u^\varepsilon-u^\varepsilon_1)u^\varepsilon_1\,dx+\frac{1}{2}\int_0^{x_1-\varepsilon l_1'}\lambda|u^\varepsilon_1|^2\,dx \\
&+\int_{x_1+\varepsilon l_1'}^1\lambda(u^\varepsilon-u^\varepsilon_2)u^\varepsilon_2\,dx+\frac{1}{2}\int_{x_1+\varepsilon l_1'}^1\lambda|u^\varepsilon_2|^2\,dx\\
&-\int_0^{x_1-\varepsilon l_1'}gu^\varepsilon\,dx+\int_0^{x_1-\varepsilon l_1'}g_1u^\varepsilon_1\,dx-\int_0^{x_1-\varepsilon l_1'}g_1u^\varepsilon_1\,dx-\int_{x_1-\varepsilon l_1'}^{x_1+\varepsilon l_1'} g u^\varepsilon\,dx\\
&-\int_{x_1+\varepsilon l_1'}^1gu^\varepsilon\,dx+\int_{x_1+\varepsilon l_1'}^1g_2u^\varepsilon_2\,dx-\int_{x_1+\varepsilon l_1'}^1g_2u^\varepsilon_2\,dx.
\end{align*}
But

\begin{align*}
-\int_0^{x_1-\varepsilon l_1'}gu^\varepsilon\,dx&+\int_0^{x_1-\varepsilon l_1'}g_1u^\varepsilon_1\,dx =\int_0^{x_1-\varepsilon l_1'}g(u^\varepsilon_1-u^\varepsilon)\,dx+\int_0^{x_1-\varepsilon l_1'}(g_1-g)u^\varepsilon_1\,dx\\
&=\int_0^{x_1-\varepsilon l_1'}g(u^\varepsilon_1-u^\varepsilon)\,dx-\int_{x_1-\varepsilon l_1'}^{x_1}(g_1-g)u^\varepsilon_1\,dx+u^\varepsilon_1\int_{0}^{x_1}(g_1-g)\,dx\\
&=\int_0^{x_1-\varepsilon l_1'}g(u^\varepsilon_1-u^\varepsilon)\,dx-\int_{x_1-\varepsilon l_1'}^{x_1}(g_1-g)u^\varepsilon_1\,dx.
\end{align*}
In the same way,

$$
-\int_{x_1+\varepsilon l_1'}^1gu^\varepsilon\,dx+\int_{x_1+\varepsilon l_1'}^1g_2u^\varepsilon_2\,dx=\int_{x_1+\varepsilon l_1'}^1g(u^\varepsilon_2-u^\varepsilon)\,dx-\int_{x_1}^{x_1+\varepsilon l_1'}(g_2-g)u^\varepsilon_2\,dx.
$$
Therefore

\begin{align*}
\mu_\varepsilon &\geq\frac{1}{2}\frac{a_1}{2l_1'}[u^\varepsilon(x_1+\varepsilon l_1')-u^\varepsilon(x_1-\varepsilon l_1')]^2+\frac{1}{2}\lambda|u^\varepsilon_1|^2(x_1-\varepsilon l_1')+\frac{1}{2}\lambda|u^\varepsilon_2|^2(1-x_1-\varepsilon l_1')\\
&+\int_0^{x_1-\varepsilon l_1'}\lambda(u^\varepsilon-u^\varepsilon_1)u^\varepsilon_1\,dx+\int_{x_1+\varepsilon l_1'}^1\lambda(u^\varepsilon-u^\varepsilon_2)u^\varepsilon_2\,dx-g_1 u^\varepsilon_1(x_1-\varepsilon l_1')-g_2u^\varepsilon_2(1-x_1-\varepsilon l_1')\\
&+\int_0^{x_1-\varepsilon l_1'}g(u^\varepsilon_1-u^\varepsilon)\,dx-\int_{x_1-\varepsilon l_1'}^{x_1}(g_1-g)u^\varepsilon_1\,dx-\int_{x_1-\varepsilon l_1'}^{x_1+\varepsilon l_1'} g u^\varepsilon\,dx\\
&+\int_{x_1+\varepsilon l_1'}^1g(u^\varepsilon_2-u^\varepsilon)\,dx-\int_{x_1}^{x_1+\varepsilon l_1'}(g_2-g)u^\varepsilon_2\,dx\\
&=\frac{1}{2}\frac{a_1}{2l_1}(u^\varepsilon_2-u^\varepsilon_1)^2+\frac{1}{2}\lambda|u^\varepsilon_1|^2x_1+\frac{1}{2}\lambda|u^\varepsilon_2|^2(1-x_1)-g_1u^\varepsilon_1x_1-g_2u^\varepsilon_2(1-x_1)\\
&-\frac{1}{2}\frac{a_1}{2l_1}(u^\varepsilon_2-u^\varepsilon_1)^2+\frac{1}{2}\frac{a_1}{2l_1'}[u^\varepsilon(x_1+\varepsilon l_1')-u^\varepsilon(x_1-\varepsilon l_1')]^2-\frac{1}{2}\lambda|u^\varepsilon_1|^2\varepsilon l_1'-\frac{1}{2}\lambda|u^\varepsilon_2|^2\varepsilon l_1'\\
&+\int_0^{x_1-\varepsilon l_1'}\lambda(u^\varepsilon-u^\varepsilon_1)u^\varepsilon_1\,dx+\int_{x_1+\varepsilon l_1'}^1\lambda(u^\varepsilon-u^\varepsilon_2)u^\varepsilon_2\,dx+g_1u^\varepsilon_1\varepsilon l_1'+g_2u^\varepsilon_2\varepsilon l_1'\\
&+\int_0^{x_1-\varepsilon l_1'}g(u^\varepsilon_1-u^\varepsilon)\,dx-\int_{x_1-\varepsilon l_1'}^{x_1}(g_1-g)u^\varepsilon_1\,dx-\int_{x_1-\varepsilon l_1'}^{x_1+\varepsilon l_1'} g u^\varepsilon\,dx\\
&+\int_{x_1+\varepsilon l_1'}^1g(u^\varepsilon_2-u^\varepsilon)\,dx-\int_{x_1}^{x_1+\varepsilon l_1'}(g_2-g)u^\varepsilon_2\,dx.
\end{align*}
But
\begin{align*}
\frac{1}{2}\frac{a_1}{2l_1'}[u^\varepsilon(x_1+\varepsilon l_1')-u^\varepsilon(x_1-\varepsilon l_1')]^2&\geq  \frac{1}{2}\frac{a_1}{2l_1'}[2(u^\varepsilon(x_1+\varepsilon l_1')-u^\varepsilon_2+u^\varepsilon_1-u^\varepsilon(x_1-\varepsilon l_1'))(u^\varepsilon_2-u^\varepsilon_1)]\\
&+ \frac{1}{2}\frac{a_1}{2l_1'}(u^\varepsilon_2-u^\varepsilon_1)^2.
\end{align*}
Thus
\begin{align*}
\mu_0-\mu_\varepsilon &\leq C[|l_1-l_1'|+\varepsilon^\frac{1}{2}+|u^\varepsilon(x_1+\varepsilon l_1')-u^\varepsilon_2|+|u^\varepsilon_1-u^\varepsilon(x_1-\varepsilon l_1')|]\\
&+C(\|u^\varepsilon_1-u^\varepsilon\|_{L^\infty(0,x_1)}+\|u^\varepsilon_2-u^\varepsilon\|_{L^\infty(x_1,1)}).
\end{align*}
By Poincar\'e's inequality for average, we have
$$
\|u^\varepsilon_1-u^\varepsilon\|_{L^\infty(0,x_1)}\leq \|u^\varepsilon_x\|_{L^1(0,x_1)}= \int_0^{x_1}|u^\varepsilon_x|\,dx=\int_0^{x_1-\varepsilon l_1'}|u^\varepsilon_x|\,dx+\int_{x_1-\varepsilon l_1'}^{x_1}|u^\varepsilon_x|\,dx\leq C\varepsilon^\frac{1}{2},
$$
where we have used the H\"{o}lder's  inequality and the estimate
\begin{equation}\label{estimate_large_diffusion}
\frac{e_1}{\varepsilon}\int_0^{x_1-\varepsilon l_1'}|u^\varepsilon_x|^2\,dx \leq\int_0^{x_1-\varepsilon l_1'}p_\varepsilon |u^\varepsilon_x|^2\,dx\leq C.
\end{equation}
In the same way, $\|u^\varepsilon_1-u^\varepsilon\|_{L^\infty(x_1,1)}\leq C\varepsilon^\frac{1}{2}$.

Now note that
\begin{align*}
|u^\varepsilon_1-u^\varepsilon(x_1-\varepsilon l_1')|&\leq \frac{1}{x_1}\int_{0}^{x_1}|u^\varepsilon-u^\varepsilon(x_1-\varepsilon l_1')|\,dx \\
& = \frac{1}{x_1}\Big[\int_{0}^{x_1-\varepsilon l_1'}+\int_{x_1-\varepsilon l_1'}^{x_1}|u^\varepsilon-u^\varepsilon(x_1-\varepsilon l_1')|\,dx\Big]\\
&\leq \int_{0}^{x_1-\varepsilon l_1'}|u^\varepsilon-u^\varepsilon(x_1-\varepsilon l_1')|\,dx+C\varepsilon^\frac{1}{2}.
\end{align*}
But, for $x\in [0,x_1-\varepsilon l_1']$, it follows by \eqref{estimate_large_diffusion} that
$$
|u^\varepsilon(x)-u^\varepsilon(x_1-\varepsilon l_1')|\leq \int_{0}^{x_1-\varepsilon l_1'}|u^\varepsilon_x|\,dx\leq \Big(\int_{0}^{x_1-\varepsilon l_1'}|u^\varepsilon_x|^2\,dx \Big)^\frac{1}{2}\leq C\varepsilon^\frac{1}{2}.
$$
Therefore $|u^\varepsilon(x)-u^\varepsilon(x_1-\varepsilon l_1')|\leq C\varepsilon^\frac{1}{2}$. In the same way $|u^\varepsilon_2-u^\varepsilon(x_1+\varepsilon l_1')|\leq C\varepsilon^\frac{1}{2}.$ 

Putting all estimates together we obtain \eqref{convergence_resolvent0}.
\cqd
\end{proof}

As a consequence of the previous result, we have the following result.

\begin{corollary}
There is a positive constant $C$ independent of $\varepsilon$ such that
\begin{equation}\label{estimate_resolvent_ODE}
\|A_\varepsilon^{-1}-E_\varepsilon A_0^{-1}M_\varepsilon \|_{\LL(L^2,X_\varepsilon^\frac{1}{2})}^2\leq C(|a_1'-a_1|+|l_1'-l_1|+ \varepsilon^\frac{1}{2}).
\end{equation}
\end{corollary}

Finally we can estimate the convergence of the attractors $\mathcal{A}_\varepsilon$ of \eqref{coupled_semilinear_problem}. 

\begin{theorem}
There is an invariant manifold $\mathcal{M}_\varepsilon$ for \eqref{coupled_semilinear_problem} such that $\mathcal{A}_\varepsilon\subset \mathcal{M}_\varepsilon$ and the flow on $\mathcal{A}_\varepsilon$ can be reduced to an ordinary differential equation. Moreover the continuity of attractors of \eqref{coupled_semilinear_problem} can be estimate by   
\begin{equation}\label{estimate_attractor_final_ODE}
\tn{d}_\varepsilon(\mathcal{A}_\varepsilon,E_\varepsilon \mathcal{A}_0)\leq C(|a_1'-a_1|+|l_1'-l_1|+ \varepsilon^\frac{1}{2})^\frac{1}{2}.
\end{equation}
\end{theorem}
\begin{proof}
Since $\{A_\varepsilon\}_{\varepsilon\in [0,\varepsilon_0]}$ is a family of positive self adjoint linear operator such that \eqref{estimate_resolvent_ODE} is valid, we can apply the Sections \ref{Functional Setting}, \ref{Compact Convergence}, \ref{Invariant Manifold}, \ref{Rate of Convergence}  and \ref{Shadowing Theory and Rate of Convergence} to obtain the estimate \eqref{estimate_attractor_final_ODE}.  
\cqd
\end{proof}

%%%%%%%%%%%%%%%%%%%%%%%%%%%%%%%%%%%%%%%%%%%%%%%%%%%%%%%%%%%%%%%%%%%%%%%%%%%%%%%%%%%%%%%%%%%%%%%%%%%
%\clearpage\thispagestyle{empty}%\cleardoublepage
\bibliographystyle{abbrv}
\bibliography{../Jabref/References}
%%%%%%%%%%%%%%%%%%%%%%%%%%%%%%%%%%%%%%%%%%%%%%%%%%%%%%%%%%%%%%%%%%%%%%%%%%%%%%%%%%%%%%%%%%%%%%%%%%%%%%
\end{document}